\documentclass[11pt]{amsart}

\usepackage{amsmath,amsthm,amsfonts,amssymb,mathrsfs,bbm}
\usepackage{mathtools}
\usepackage{enumitem}
\usepackage{float}
\usepackage{braket}
\usepackage[margin=1.4in]{geometry}
\usepackage[textwidth=1.3in,textsize=tiny]{todonotes}

\usepackage{tikz}
\usetikzlibrary{arrows,matrix, cd}
\usepackage{tikz-cd}

\def\inn#1#2{\langle #1, #2 \rangle}

\def\cksigma{\check{\sigma}}

\newcommand{\Orb}{\mathbb O}

\newcommand{\C}{\mathbb{C}}
\newcommand{\R}{\mathbb{R}}
\newcommand{\Z}{\mathbb{Z}}

\DeclareMathOperator{\Or}{O}

\def\Pr{\mathrm{Pr}}

\newcommand{\ckfg}{{\check{\fg}}}
\newcommand{\ckfm}{{\check{\fm}}}
\newcommand{\ckfp}{{\check{\fp}}}
\newcommand{\ckfu}{{\check{\fu}}}
\newcommand{\ckfh}{{\check{\fh}}}
\newcommand{\eqL}{\approx_{L}}
\newcommand{\eqR}{\approx_{R}}
\newcommand{\eqLR}{\approx_{LR}}
\newcommand{\leqL}{\leq_{L}}
\newcommand{\leqR}{\leq_{R}}
\newcommand{\leqLR}{\leq_{LR}}
\newcommand{\bfzero}{\mathbf{0}}
\newcommand{\bfone}{\mathbf{1}}

\newcommand{\bfd}{\mathbf{d}}
\newcommand{\bfp}{\mathbf{p}}
\newcommand{\bfq}{\mathbf{q}}
\newcommand{\bfr}{\mathbf{r}}
\newcommand{\bfs}{\mathbf{s}}

\newcommand{\half}{\tfrac{1}{2}}

\newcommand{\fb}{\mathfrak{b}}

\newcommand{\fh}{\mathfrak{h}}
\newcommand{\fu}{\mathfrak{u}}

\newcommand{\fg}{\mathfrak{g}}

\newcommand{\fl}{\mathfrak{l}}
\newcommand{\fm}{\mathfrak{m}}

\newcommand{\fp}{\mathfrak{p}}

\newcommand{\cK}{\mathcal{K}}

\newcommand{\cN}{\mathcal{N}}

\newcommand{\cV}{\mathcal{V}}

\newcommand{\cZ}{\mathcal{Z}}

\newcommand{\cU}{\mathcal{U}}

\newcommand{\ckalpha}{\check{\alpha}}

 \DeclareMathOperator{\ad}{ad}
 
\DeclareMathOperator{\Irr}{Irr}

\DeclareMathOperator{\Ind}{Ind}
\DeclareMathOperator{\Id}{Id}
\DeclareMathOperator{\wt}{wt}

\newcommand{\nilcone}{{\mathcal N}}

\newcommand{\parti}{\mathcal P}
\newcommand{\g}{{\mathfrak g}}
\newcommand{\Ug}{\cU\fg}
\newcommand{\Zg}{\cZ\fg}

\DeclareMathOperator{\GL}{GL}

\DeclareMathOperator{\SL}{SL}
\DeclareMathOperator{\Sp}{Sp}
\DeclareMathOperator{\SO}{SO}
\DeclareMathOperator{\Spin}{Spin}
\DeclareMathOperator{\so}{\mathfrak{so}}
\DeclareMathOperator{\gl}{\mathfrak{gl}}
\DeclareMathOperator{\sgn}{sgn}
\DeclareMathOperator{\sg}{sg}
\DeclareMathOperator{\spr}{Spr}

\DeclareMathOperator{\Coh}{Coh}
\DeclareMathOperator{\Ann}{Ann}

\DeclareMathOperator{\spn}{Span}

 \newcommand{\renc}{\renewcommand}
\renc{\sl}{{\mathfrak{sl}}}
\renc{\sp}{{\mathfrak{sp}}}   

\newcommand{\trivial}[2][]{\if\relax\detokenize{#1}\relax{\color{orange} \vspace{0em} #2} \else\ifx#1h\ifcsname showtrivial\endcsname {\color{orange} \vspace{0em} #2} \fi\else{\color{red} Wrong argument!}\fi\relax\fi}

\RequirePackage{ae, aecompl, aeguill} 
\RequirePackage{color}
\definecolor{myred}{rgb}{0.75,0,0}
\definecolor{mygreen}{rgb}{0,0.5,0}
\definecolor{myblue}{rgb}{0,0,0.65}

\usepackage{hyperref}

\usepackage[capitalise,noabbrev,nameinlink]{cleveref}
\usepackage{comment}

\newtheorem{theorem}{Theorem}[section]
\newtheorem{lemma}[theorem]{Lemma}
\newtheorem{proposition}[theorem]{Proposition}
\newtheorem{corollary}[theorem]{Corollary}

\newtheorem{example}[theorem]{Example}

\theoremstyle{definition}
\newtheorem{definition}[theorem]{Definition}

\theoremstyle{remark}
\newtheorem{remark}[theorem]{Remark}

\def\Jmax{J_{\mathrm{max}}}

\def\Rep{\mathrm{Rep}}

\def\leqLR{\leq_{LR}}

\begin{document}

\title{Weak unipotence and Langlands duality}

\author{Jia-jun Ma}
\address{J. M.: School of Mathematical Sciences, Xiamen University, Xiamen, Fujian, China}
\email{hoxide@gmail.com}
\author{Shilin Yu}
\address{S.Y.: School of Mathematical Sciences, East China Normal University, Shanghai, China}
\address{S.Y.: School of Mathematical Sciences, Xiamen University, Xiamen, Fujian, China}
\email{turingfish@gmail.com}

\date{\today}

\begin{abstract}
	Weak unipotence of primitive ideals is a crucial property in the study of unitary representations of reductive groups. We establish a sufficient condition, referred to as \emph{mild unipotence}, which guarantees weak unipotence and is more accessible in practice. We establish mild unipotence for both the $q$-unipotent ideals defined by McGovern \cite{McGovern1994} and unipotent ideals attached to nilpotent orbit covers defined by Losev-Mason-Brown-Matvieievskyi \cite{LMBM}. Our proof is conceptual and uses the bijection between special orbits in type $D$ and metaplectic special orbits in type $C$ found in \cite{BMSZ:metaplecticBV} in an essential way.
\end{abstract}

\maketitle
\tableofcontents

\section{Introduction}

In \cite{Vogan:unitarizability}, Vogan defined the notion of weak unipotence of primitive ideals of the universal enveloping algebra of a reductive Lie algebra $\fg$, see \Cref{defn:weakly_unipotent}. This notion plays a crucial role in the study of unitary representations of reductive groups (see \cite{Vogan:unitarizability} and \cite{Davis-Mason-Brown:Hodge}). 
In \cite[Proposition~5.10]{BV85:unipotent}, Barbasch and Vogan proved weak unipotence of special unipotent ideals attached to even orbits in the Langlands dual Lie algebra $\ckfg$ of $\fg$. It is natural to ask if there are more weakly unipotent primitive ideals that arise naturally from nilpotent orbits.
In \cite{McGovern1994}, McGovern introduced the notion of $q$-unipotent infinitesimal characters and $q$-unipotent ideals for classical Lie algebras. In \cite{LMBM}, Losev, Mason-Brown, and Matvieievskyi attached unipotent ideals to nilpotent orbit covers. For linear classical groups of type $B$, $C$ and $D$, their unipotent ideals are special cases of $q$-unipotent ideals, while for type $A$, spin and exceptional groups, the unipotent ideals in \cite{LMBM} provide new examples. In all these cases, the primitive ideals are maximal primitive ideals with certain infinitesimal characters. Therefore we will also speak of weakly unipotence of an infinitesimal character, which just means the weakly unipotent of the maximal primitive ideal with this infinitesimal character.

The main result of this paper is the following theorem.

\begin{theorem}
	Let $\g$ be a complex semisimple Lie algebra.
	\begin{enumerate}[label=(\roman*), nosep]
		\item if $\g$ is of classical type, then q-unipotent infinitesimal characters (cf. \Cref{defn:q-unipotent}) and unipotent ideals attached to covers of nilpotent orbits in $\g^*$ are weakly unipotent.
		\item if $\g$ is of exceptional type, then all unipotent ideals attached to birational rigid covers of nilpotent orbits in $\g^*$ are weakly unipotent.
	\end{enumerate}
\end{theorem}

To prove the above theorem, we introduce the concept of mild unipotence (\Cref{defn:mildly_unip}) of a two-sided ideal in $\Ug$, which has already been implicitly used in \cite{BV85:unipotent}. This concept is based on the theory of cells developed by Kazhdan-Lusztig and Barbasch-Vogan (see \cite[Section~3]{BMSZ:counting} for an exposition of the theory). We then show that mild unipotence implies weak unipotence (\Cref{cor:mildly_implies_weakly}) which is equivalent to some containment conditions of nilpotent orbits in the Lie algebra attached to the dual of the integral root system of $\lambda$ (\Cref{cor:closure}). 

Then we show that McGovern's $q$-unipotent ideals are all mildly unipotent (\Cref{thm:q-unipotent_weaklyunip}). The proof is based on the bijection between special orbits in type $D$ and metaplectic special orbits in type $C$ found in \cite{BMSZ:metaplecticBV}, as well as the Springer duality between special orbits in type $B$ and type $C$. Both will be recalled in \Cref{subsec:special_partitions}. Using these bijections, we are able to reduce the case of $q$-unipotent infinitesimal characters to the case of special unipotent infinitesimal characters, for which \cite[Lemma 5.7]{BV85:unipotent} can again be applied (see \Cref{prop:norm_comparison}). Since the original proof of \cite[Lemma 5.7]{BV85:unipotent} is uniform and case-free, our proof of mild/weak unipotence of $q$-unipotent ideals is conceptually simple and does not require any complicated combinatorial computation.

We also study more general mildly unipotent infinitesimal characters in the case of type $A$ in \Cref{subsec:type_A}, for which a light combinatorial computation is needed. Combining with the case of $q$-unipotent infinitesimal characters, we can prove mild/weak unipotence of a larger class of infinitesimal characters in \Cref{subsec:general_cases}, which can be regarded as deformations of $q$-unipotent infinitesimal characters. This allows us to prove mild/weak unipotence of all unipotent ideals attached to nilpotent orbit covers of classical Lie algebras in \Cref{cor:cover_weaklyunip_classical}, which includes the case of covers of orthogonal Lie algebras $\so(N)$ that are equivariant under the spin groups $\Spin(N)$ but not the orthogonal groups $\SO(N)$.

The case of exceptional groups is analyzed in \Cref{sec:exceptional} and uses the \texttt{atlas} software \cite{atlas}.

{\bf Acknowledgements}: We would like to thank Dougal Davis, Lucas Mason-Brown and William McGovern for helpful discussions.

The work of S.Y. has been partially supported by China NSFC grants (Grant No. 12001453 and 12131018) and Natural Science Foundation of Fujian Province (Grant No. 2022J06005).

\section{Weak and mild unipotence}
\label{sec:weak_and_mild_unipotence}

\subsection{Weak unipotence}

Let $\g$ be a complex reductive Lie algebra and $\Ug$ be the universal enveloping algebra of $\g$.  
By the Harish-Chandra isomorphism, the center $\Zg$ of $\Ug$ is isomorphic to the algebra $(S \fh)^W \simeq \C[\fh^*/W]$ of $W$-invariant polynomials on $\fh^*$, where $\fh$ is the abstract Cartan subalgebra of $\g$ and $W$ is the abstract Weyl group. We identify the set of infinitesimal characters of $\Ug$ with $\fh^*/W$ and use $\chi_\lambda$ to denote the infinitesimal character of $\Ug$ corresponding to an orbit $W\cdot\lambda$ in $\fh^*/W$.
It is well known that there is a maximal primitive ideal $\Jmax(\lambda)$ of $\Ug$ for each infinitesimal character $\chi_\lambda$. Let $G$ be a connected complex reductive group with Lie algebra $\g$. Let $X^* \subset \fh^*$ be the corresponding weight lattice of $G$ and $\fh^*_\R \subset \fh^*$ be the real span of $X^*$. We fix a $W$-invariant inner product on $\fh^*_\R$ and write $\lVert\cdot\rVert$ for the associated norm.

Suppose $M$ is a $\g$-module and $\gamma \in \fh^*/W$. We define
\begin{equation}\label{eq:Pr_gamma}
{\Pr}_\gamma(M) := \{ m \in M | \forall z \in \cZ(\g), (z - \chi_\gamma(z))^k m = 0 \text{ for some positive integer } k \}
\end{equation}
to be the generalized eigenspace of $M$ with respect to the infinitesimal character $\chi_\gamma$.
Now suppose $M$ has generalized infinitesimal character $\chi_\lambda$. For any finite-dimensional representation $F$ of $G$, we can form the tensor product $M \otimes_\C F = M \otimes F$ which is again a $\g$-module. By \cite[Corollary 7.1.13]{Vogan:greenbook} (cf. \cite[Theorem 5.1]{Kostant:tensor}) we have a (finite) direct sum decomposition
\[M \otimes F = \bigoplus_{\gamma \in \fh^*/W} {\Pr}_\gamma (M \otimes F). \]
In fact, \cite[Corollary 7.1.13]{Vogan:greenbook} says that ${\Pr}_\gamma(M \otimes F) \neq 0$ only if $\gamma = \lambda + \mu$ for some weight $\mu$ of $F$.


\begin{definition}[{c.f. \cite[Definition 8.16]{Vogan:unitarizability}}] \label{defn:weakly_unipotent}
Let $M$ be a $\g$-module with generalized infinitesimal character $\chi_\lambda$ for $\lambda \in \fh_\R^*$. We say that $M$ is \emph{weakly unipotent} with respect to $G$ (or its weight lattice $X^*$) if, for any finite-dimensional representation $F$ of $G$ (or equivalently, $F$ with weights in $X^*$) and any $\nu \in \fh^*/W$ such that $\lVert \nu \rVert < \lVert \lambda \rVert$, we always have ${\Pr}_\nu(M \otimes F)=0$.

We call $\chi_\lambda$ a \emph{weakly unipotent infinitesimal character} if $\Jmax(\lambda)$ is weakly unipotent. 
\end{definition}

When $\g$ is semisimple, we will often take $G$ to the adjoint group $G_{ad}$ of $\g$. There are primitive ideals that are weakly unipotent with respect to root lattices but not weight lattices, see \Cref{ex:root_vs_weight}.

The following result is the analogue of \cite[Proposition 4.13]{Davis-Mason-Brown:Hodge} and the proof is exactly the same. It means that the notion of weak unipotence in fact only depends on the annihilator ideal of the $\Ug$-module. Therefore we can talk about weak unipotence of a two-sided ideal in $\Ug$.

\begin{proposition} \label{prop:weakly unipotent ideal}
Fix $\lambda \in \fh^*_\R$ and let $I \subset \Ug$ be a two-sided ideal with generalized infinitesimal character $\chi_\lambda$. Let $G$ be a connected algebraic group with Lie algebra $\g$.
Then the following conditions are equivalent.
\begin{enumerate}[label=(\roman*), nosep]
\item \label{itm:weakly unipotent ideal 1} 
Every $\Ug$-module annihilated by $I$ is weakly unipotent with respect to $G$.
\item \label{itm:weakly unipotent ideal 2} 
There exists a $\g$-module $M$ with $\Ann_{\Ug}(M) = I$ such that $M$ is weakly unipotent with respect to $G$.
\item \label{itm:weakly unipotent ideal 3} 
If $\lVert \nu \rVert < \lVert \lambda \rVert$ for some $\nu \in \fh^*_\R$ and $F$ is a finite dimensional algebraic representation of $G$, then $\Ann_{\cZ(\g)}(\Ug/I \otimes F) \not\subset \ker \chi_\nu$.
\end{enumerate}
\end{proposition}

\subsection{Coherent families and Goldie rank representations}

Coherent families of group representations were introduced by Schmid \cite{Schmid:character}. See also \cite{Zuckerman:tensor} and \cite{Speh-Vogan:reducibility}. 
We refer the reader to \cite[Chapter 7]{Vogan:greenbook} as a general reference for coherent families in this setting. We refer the reader to \cite[\S\,3 \& \S\,4]{BMSZ:counting} for the notations.


Fix a coset $\Lambda \in \fh^*/X^*$ where $\fh$ and let $Q$ be the root lattice of $\fg$. 
The integral Weyl group 
 $W(\Lambda)$ is defined by  
\[
W(\Lambda) := \{ w\in W |\langle{\lambda - w \lambda,}{\ckalpha}\rangle \in \mathbb{Z}, \text{ for all } \lambda \in \Lambda \text{ and coroot } \ckalpha\}
= \{ w\in W |\lambda - w \lambda \in Q\}. 
\]
It is known that $W(\Lambda)$ is the Weyl group of a root system 
\[
R(\Lambda) := \{\alpha \in \Delta(\fg,\fh) | \langle \ckalpha, \lambda \rangle \in \mathbb{Z} \text{ for all } \lambda \in \Lambda\}.
\]
Fix a positive system $R^+(\Lambda)$ of $R(\Lambda)$. 
We call an element $\lambda \in \Lambda$ \emph{dominant} if $\inn{\lambda}{\ckalpha} \geq 0$ for all $\alpha \in R^+(\Lambda)$.
Let $\Lambda^+$ be the cone of all dominant elements in $\Lambda$. 
On the other hand, we define 
\[
W_\Lambda := \{w \in W | w \lambda \in \Lambda \text{ for all } \lambda \in \Lambda\} = \{w \in W | \lambda - w \lambda \in X^*\}
\]

Therefore, the coset $W_\Lambda/W(\Lambda)$ space has a set of representatives 
 \[ B := \set{b_1, \cdots, b_k} \] 
 such that $b_i R^+(\Lambda) = R^+(\Lambda)$ for all $i = 1, \cdots, k$. 

We record the following easy lemma. 
\begin{lemma}
\label{lem:dominant}
Let $\lambda \in \Lambda$ be a dominant element. Then $b^{-1} \lambda$ is also dominant for every $b \in B$.  
\end{lemma}
\begin{proof}
	Let $\alpha \in R^+(\Lambda)$. Then $b\alpha \in R^+(\Lambda)$ by the definition of $B$.  
	Therefore,
	\[
     \inn{b^{-1} \lambda}{\ckalpha} = \inn{\lambda}{b\ckalpha} \geq 0.
     \qedhere
	\]
\end{proof}
%
%
%
%
%
%
%
%


We fix a Borel subalgebra $\fb$ of $\fg$ and consider the category $\Rep(\fg,\fb)$ of finitely generated $\fg$-modules that are unions of finite-dimensional $\fb$-submodules. Let $\cK(\fg,\fb)$ be the Grothendieck group of $\Rep(\fg,\fb)$.
Let $\Coh_{\Lambda}(\cK(\g,\fb))$ be the space of coherent families with value in $\cK(\fg,\fb)\otimes_\Z \C$ based on $\Lambda$, see \cite[Definition~3.3]{BMSZ:counting}. It is a $W_\Lambda$-module with the action given by $w \cdot \Psi(\nu) = \Psi(w^{-1} \nu)$ for all $\Psi\in \Coh_{\Lambda}(\cK(\g,\fb))$, $\nu \in \Lambda$ and $w \in W_\Lambda$. We restrict the action to $W(\Lambda)$ and therefore regard it as a $W(\Lambda)$-action. Moreover, there is also a $W$-action on $\Coh_{\Lambda}(\cK(\g,\fb))$ commuting with the $W_\Lambda$-action. In summary, $\Coh_{\Lambda}(\cK(\g,\fb))$ is a $W\times W(\Lambda)$-module. See \cite[Section 3.2]{BMSZ:counting}.

An element $\Psi\in \Coh_{\Lambda}(\cK(\g,\fb))$ is called \emph{basal} if $\Psi(\nu)$ is either zero or irreducible for all $\nu \in \Lambda^+$.
For a basal element, let $\braket{\Psi}_{L}$ (resp. $\braket{\Psi}_{R}$, $\braket{\Psi}_{LR}$) be the smallest  $W$-invariant (resp. $W(\Lambda)$-invariant, $W\times W(\Lambda)$-invariant) basal subspace of $\Coh_{\Lambda}(\cK(\g,\fb))$ containing $\Psi$, called the left (resp. right, two-sided) cone (representation) of $\Psi$. This defines left (resp. right, two-sided) preorders $\leqL$ (resp. $\leqR$, $\leqLR$) on the set of basal elements. The equivalence relations associated to these preorders are denoted by $\eqL$, $\eqR$ and $\eqLR$ (we will mostly use $\eqLR$ only). The equivalence classes are called left, right and two-sided cells respectively. Relevant to cone representations, we also have the notions of left, right and two-sided cell representations. The notion of two-sided cell representations induces a partial order $\leqLR$ and equivlanece relation $\eqLR$ on the set $\Irr(W(\Lambda)) = \widehat{W(\Lambda)}$ of isomorphism classes of all irreducible representations of $W(\Lambda)$. A equivalence class for $\eqLR$ on $\Irr(W(\Lambda))$ is called a two-sided cell of $\Irr(W(\Lambda))$. It is known that each two-sided cell of $\Irr(W(\Lambda))$ contains a unique Goldie rank/special representation of $W(\Lambda)$, hence gives a bijection between two-sided cells and special representations of $W(\Lambda)$.
We refer the reader to \cite[Section 5]{Joseph:Goldie_II} and \cite[Section 2]{BV83:exceptional}, as well asSection 3.4, Definition 3.20, Proposition 3.22 of \cite{BMSZ:counting} for more details.

Let $p_{J}$ denote the Goldie rank polynomial of the primitive ideal $J$. Joseph \cite{Joseph:Goldie_II} showed that $\sigma_J := \spn \Set{W(\Lambda) p_{J}}$ is an irreducible $W(\Lambda)$-subrepresentation occurring in $S(\fh)$, called the Goldie rank representation attached to $J$. By \cite[Corollary 2.16]{BV83:exceptional}, each double cell in $\Irr(W(\Lambda))$ contains a unique Goldie rank representation of $W(\Lambda)$.
Moreover, by \cite[Theorem 2.29]{BV83:exceptional}, $\sigma_J$ is a \emph{special representation} of $W(\Lambda)$ in the sense of Lusztig (\cite{Lusztig:irred_Weyl_I,Lusztig:irred_Weyl_II}), which means that its fake degree is the same as its generic degree. Let $\Irr(W(\Lambda))^{sp} \subset \Irr(W(\Lambda))$ denote the set of isomorphism classes of all special representations of $W(\Lambda)$. 
By Duflo \cite[Theorem 1]{Duflo:primitive}, all primitive ideals of infinitesimal character $\lambda$ are of the form $\Ann(L(w \lambda))$ for some $w \in W(\Lambda)$.
For $w_1, w_2 \in W(\Lambda)$, $\sigma_{\Ann(L(w_1\lambda))} = \sigma_{\Ann(L(w_2\lambda))}$ if and only if $w_1$ and $w_2$ are in the same double cell of $W(\Lambda)$, see \cite[Proposition~2.28]{BV83:exceptional}. 

\subsection{Special nilpotent orbits and dualities} \label{subsec:special_duality}

Let $\g$ be a reductive Lie algebra over $\C$ with Weyl group $W$, and $\ckfg$ be its Langlands dual Lie algebra with the same Weyl group $W$. The $\ckfh$ be the abstract Cartan subalgebra of $\ckfg$, which is identified as the linear dual $\fh^*$ of the abstract Cartan subalgebra $\fh$ of $\g$. Let $\cN_o$ and $\check{\cN}_o$ denote the set of all nilpotent orbits in $\g^*$ and $\ckfg^*$, respectively. For any $\Orb \in \cN_o$, let $\mathrm{Loc}(\Orb)$ be the set of isomnorphism classes of irreducible $G_{ad}$-equivariant local systems over $\Orb$. The Springer correspondence is an injective map
\[ 
\spr_\g : \widehat{W} \hookrightarrow \{ (\Orb, \rho) \,|\, \Orb \in \cN_o, \, \rho \in \mathrm{Loc}(\Orb) \}. 
\]
We will also abbreviate $\spr_\g$ to $\spr$ when the relevant Lie algebra $\g$ is clear from the context.
For any special representation $\sigma$, it is known that $\spr_\g(\sigma)$ is of the form $(\Orb_\sigma, \mathbbm{1})$, where $\mathbbm{1}$ stands for the trivial local system over $\Orb_\sigma$ (\cite{Lusztig:irred_Weyl_I}). We say that $\Orb_\sigma$ is the \emph{special orbit} associated with the special representation $\sigma$. Let $\cN_o^{sp}$ denote the set of all special nilpotent orbits in $\g^*$. Then we can regard $\spr_\g$ as a bijection
\begin{equation} \label{eq:springer_special} 
\spr_\g: \Irr(W)^{sp} \xrightarrow{\sim} \cN_o^{sp}. 
\end{equation}
Note that the notion of special representations is intrinsic to the Weyl group $W$, hence we also have 
\begin{equation} \label{eq:special_intrinsic}
	\spr_{\ckfg}: \Irr(W)^{sp} \xrightarrow{\sim} \check{\cN}_o^{sp},
\end{equation}
where $\check{\cN}_o^{sp}$ is the set of all special nilpotent orbits in $\ckfg^*$.

By \cite[Proposition 3.24]{BV85:unipotent}, there is an order-reversing involution on the set of two-sided cells in $W$, which we denote by $d$. Since there is a bijection between two-sided cells and special representations of $W$, we also have an order-reversing involution on the set of special representations of $W$, which we also denote by $d(\sigma) = \check{\sigma}$. At the level of special representations, this involution is given by sending a special representation $\sigma$ to the unique special representation $\check{\sigma}$ in the same double cell as $\sigma \otimes \sgn$, where $\sgn$ denotes the sign representation of $W$. Therefore we have $\check{\sigma} \eqLR \sigma \otimes \sgn$. Note that in classical case, $\sigma \otimes \sgn$ is always a special, so $\check{\sigma} = \sigma \otimes \sgn$. This is almost also true for exceptional types, except for three cases in type $E_7$ and $E_8$ (see \cite[Definition 4.5]{BV85:unipotent}).

Combing with \eqref{eq:springer_special}, we have an order-reversing involution $\cN_o^{sp}$, found by Lusztig and Spaltenstein (\cite{Spaltenstein}),
\begin{equation} \label{eq:LS_duality}
	d_{LS} = \spr_\g \circ \,d \circ \spr_\g^{-1}: \cN_o^{sp} \to \cN_o^{sp}, \quad \Orb_\sigma \mapsto \spr_\g(\Orb_{\check{\sigma}}).
\end{equation}
Barbasch and Vogan \cite{BV85:unipotent} also defined a second order-reversing bijection
\begin{equation} \label{eq:BV_duality}
	d_{BV} = \spr_{\fg} \circ \,d \circ \spr_{\ckfg}^{-1}: \check{\cN}_o^{sp} \to \cN_o^{sp}, \quad \Orb_\sigma \mapsto \spr_{\ckfg}(\Orb_{\check{\sigma}}).
\end{equation}
Moreover, $d_{BV}$ can be extended to an order-reversing map 
\begin{equation} \label{eq:BV_duality_full}
	d_{BV}: \check{\cN}_o \to \cN_o^{sp} \subset \cN_o
\end{equation}
as follows. Given any $\check{\Orb} \in \check{\cN}_o$, we choose a point $e \in \check{\Orb}$. By Jacobson-Morozov theorem, $e$ can be completed into an $\sl_2$-triple $(e,h,f)$ with $f$ nilpotent and the semisimple element $h$ lying in $\ckfh$. The element $h$ is uniquely determined up to conjugation by $W$ and is independent of the choice of $e$ or the $\sl_2$-triple. We write $h_{\check{\Orb}} = h  \in \ckfh/W$ and define $\lambda_{\check{\Orb}} = \frac{1}{2} h_{\check{\Orb}} \in \ckfh/W$. Since $\ckfh \simeq \fh^*$, we can regard $\lambda_{\check{\Orb}} \in \fh^*/W$ as a character of the center $\Zg$ of the universal enveloping algebra $\Ug$ of $\fg$. We consider the associated variety $\cV(\Jmax(\lambda_{\check{\Orb}}))$ of the maximal primitive ideal $J_{max}(\lambda_{\check{\Orb}})$ of $\Ug$ with infinitesimal character $\lambda_{\check{\Orb}}$. By \cite[Theorem 3.10]{Jospeh:associated_variety}, the associated variety of any primitive ideal in $\Ug$ is the Zariski closure of a unique nilpotent orbit $\Orb$ in $\g^*$. We define $d_{BV}(\check{\Orb}) = \Orb$. It is shown in \cite[Proposition A2]{BV85:unipotent} that the image of $d_{BV}$ is precisely $\cN_o^{sp}$. Following \cite{BV85:unipotent}, we call $\lambda_{\check{\Orb}}$ and $\Jmax(\lambda_{\check{\Orb}})$ the special unipotent infinitesimal character and special unipotent ideal, respectively, attached to $\check{\Orb}$.


\subsection{Mildly unipotence}
Let $J$ be a primitive ideal of $\Ug$ with infinitesimal character $\chi_\lambda$. 
Let $\sigma_J$ be the Goldie rank representation of $W(\Lambda)$ attached to $J$, which is also the special representation attached to the double cell in $W(\Lambda)$ containing $w$. We also write $\sigma_\lambda:=\sigma_{J_{max}(\lambda)}$ for the maximal primitive ideal $J_{max}(\lambda)$ with infinitesimal character $\chi_\lambda$. 
Fix a Cartan subalgebra $\fh$ of $\fg$. For any finite dimensional representation $F$ of $\fg$, let $\wt(F) \subset \fh^*$ be the set of all weights of $F$ with respect to $\fh$.

We now introduce a variant of the notion of weak unipotence for primitive ideals, which will be justified by \Cref{lem:translation_nonvanishing} and \Cref{cor:mildly_implies_weakly} below.

\begin{definition} \label{defn:mildly_unip}
	Let $\lambda \in \fh_\R^*$ and $J$ be a primitive ideal of $\Ug$ with infinitesimal character $\lambda$. Let $\sigma_{J}$ be the Goldie rank representation of $W(\Lambda)$ attached to $J$. We say that $J$ is \emph{mildly unipotent} with respect to an algebraic group $G$ (or its weight lattice $X^*$) with Lie algebra $\g$, if whenever $\sigma_J \leqLR \sigma_\nu$ for some element $\nu \in \Lambda = \lambda + X^*$, we must have $\lVert \nu \rVert \geq  \lVert \lambda \rVert$, .
\end{definition}

\begin{lemma} \label{lem:translation_nonvanishing}
	Let $F$ be a finite dimensional representation of $\fg$ with weights in $\Lambda$. Let $\Psi$ be a basal coherent family in $\Coh_{\Lambda}(\cK(\fg,\fb))$ such that $J = \Ann(\Psi(\lambda))$.
	If $(F \otimes \Psi(\lambda))_\nu \neq 0$, then $\sigma_J \leqLR \sigma_{b^{-1}\nu}$ for some $b \in B$.
\end{lemma}

\begin{proof}
	By the definition of coherent families, we have 
	\[
		\begin{split}
		{\Pr}_\nu(F \otimes L(w \lambda)) 
		&= {\Pr}_\nu(F \otimes \Psi(\lambda)) \\
		&= \sum_{\substack{\mu \in \wt(F), \\ \mu + \lambda \in W_\Lambda \cdot \nu}} \Psi(\lambda + \mu) \\
		& = \sum_{\substack{\mu \in \wt(F),\, b\in B, \\ \lambda + \mu \in W(\Lambda) \cdot b^{-1} \cdot \nu}} \Psi(\lambda + \mu) 
		\end{split}
	\]
	If ${\Pr}_\nu(F \otimes L(w \lambda)) \neq 0$, then there exists some weight $\mu$ of $F$ such that $\Psi(\lambda + \mu) \neq 0$ with $\lambda + \mu = u^{-1} b^{-1} \cdot \nu$ for some $u \in W(\Lambda)$ and $b \in B$. This means that  
	\[ (u \cdot \Psi)(b^{-1}\nu) = \Psi(u^{-1} b^{-1}\nu) =  \Psi(\lambda + \mu) \neq 0, \]
	where $u \cdot \Psi$  is given by the right $W(\Lambda)$-action on $\Coh_{\Lambda}(\cK(\fg,\fb))$. Write $u \cdot \Psi$ as a linear combination of basal elements,
	\[ u \cdot \Psi = \sum_{i=1}^{k} a_i \Psi_{i},\]
	where $0 \neq a_i \in \C$ and $\Psi_{j} \in \Coh_{\Lambda}(\cK(\fg,\fb))$ are distinct basal element. Then there exists some $j$ such that $\Psi_{j}(b^{-1}\nu) \neq 0$. 
	Again by \cite[Corollary 3.17]{BMSZ:counting}, $\Psi_{j}(d^{-1}\nu)$ is an irreducible $\Ug$-module, with the corresponding primitive ideal denoted by $I=\Ann(\Psi_{j}(d^{-1}\nu))$ and the Goldie rank representation $\sigma_I = \sigma_{\Psi_{j}}$. 
	Note that $\Psi_{j} \in \braket{\Psi}_{R}$ by definition. In particular, $\Psi_{j} \in \braket{\Psi}_{LR}$ and hence $\sigma_J \leqLR \sigma_I$. 
	
	On the other hand, we have $I \subset J_{max}(b^{-1}\nu)$. Since $b^{-1} \nu$ is dominant, by the translation principle again (\cite[Corollary 3.17]{BMSZ:counting} and \cite[Theorem 2.12]{Borho-Jantzen}), there exists a basal element $\Psi' \in \Coh_{\Lambda}(\cK(\fg,\fb))$ such that $J_{max}(b^{-1}\nu) = \Ann(\Psi'(b^{-1}\nu))$ and $\Psi' \in \braket{\Psi_{j}}_{L}$. Therefore $\sigma_I \leqLR \sigma_{b^{-1}\nu}$.
	We thus conclude that
	\[\sigma_J \leqLR \sigma_I \leqLR \sigma_{d^{-1}\nu},\]
	which is the desired result.
\end{proof}

Now \Cref{lem:translation_nonvanishing} implies the following result. 
\begin{corollary} \label{cor:mildly_implies_weakly}
    Let $J$ be a mildly unipotent primitive ideal of $\Ug$ with respect to an algebraic group $G$ of $\fg$. Then $J$ is weakly unipotent with respect to $G$ in the sense of \Cref{defn:weakly_unipotent}.
\end{corollary}

\begin{proof}
	Assume $J$ has infinitesimal character $\chi_\lambda$ with $\lambda \in \fh_\R^*$ dominant. By the translation principle (\cite[Corollary 3.17]{BMSZ:counting}), there exists a basal coherent family $\Psi$ in $\Coh_{\Lambda}(\cK(\fg,\fb))$ such that $\Psi(\lambda)$ is an irreducible $\Ug$-module with $\Ann(\Psi(\lambda)) = J$. By \Cref{prop:weakly unipotent ideal}, we only need to prove that the irreducible $\Ug$-module $\Psi(\lambda)$ is weakly unipotent with respect to $G$. By \Cref{lem:translation_nonvanishing}, $\Pr_\nu(F \otimes \Psi(\lambda)) \neq 0$ implies that $\sigma_{J}\leqLR \sigma_{b^{-1} \nu}$ for some $b \in B$. 
	Now by \Cref{defn:mildly_unip} of mild unipotence, 
	we have $\|\nu\| = \|b^{-1} \cdot \nu\| \geq \|\lambda\|$. 
\end{proof}

\subsection{Mild unipotence via Langlands duality}

Now consider the following setting.

Let $\lambda$ be an infinitesimal character and $\Lambda = \lambda + X^*$. 
Let $\ckfg$ be the dual Lie algebra of $\g$ and
$\ckfg_\Lambda$ be the Lie algebra corresponding to the dual root datum of the integral root system of $\Lambda$. Then the Weyl group of $\ckfg_\Lambda$ is $W(\Lambda)$. 

Recall that for any element $\lambda \in \fh^* = \ckfh$, the centralizer of $\lambda$ in $\ckfg$, 
\[ \ckfg_{\lambda} := \mathfrak{z}_{\ckfg}(\lambda) = \{ \zeta \in \ckfg \ | \  [\zeta, \lambda] = 0 \} \] 
is a Levi subalgebra of $\ckfg$. Its Weyl group is $W_\lambda$, the stabilizer of $\lambda$ by the $W$-action on $\ckfh$. We have $\ckfg_{\lambda} \subset \ckfg_\Lambda$ and $W_\lambda \subset W(\Lambda)$.

Let $\sigma_\lambda$ be the Goldie rank representation attached to the maximal primitive ideal $J_{max}(\lambda)$ of $\Ug$ with infinitesimal character $\chi_\lambda$. 
Let $\sgn$ denote the sign representation of the Weyl group in question. Let $j_{W'}^W$ denote the $j$-induction of Macdonald-Lusztig-Spaltenstein (\cite[Chapter 11]{Carter}). Given any $\fg$, we always use $\bfzero$ to denote the zero nilpotent orbit in $\fg^*$. 

\begin{lemma} \label{lem:j-ind_Richardson}
	Let $\fg$ be a reductive Lie algebra over $\C$ with Weyl group $W$, and $\fm$ be its levi subalgebra with Weyl group $W_0 \subset W$. Then
		\[\spr_{\fg} \left(j_{W_0}^W \sgn \right) = \Ind_{\fm}^{\fg} \bfzero. \]
\end{lemma}

\begin{proof}
	This follows from the fact that the Springer representation associated to the $0$ orbit (with the trivial local system) is the sign representation and \cite[Proposition 1.4]{Hotta-Springer} (or more generally, \cite[Theorem 3.5]{Lusztig-Spaltenstein}).
\end{proof}

We record the following lemma.

\begin{lemma}[Barbasch-Vogan] \label{prop:special_j} 
    Let $\nu\in \Lambda$. Then we have
    \[
	   \cksigma_\nu = j_{W_\nu}^{W(\Lambda)} \sgn.
	\]
\end{lemma}

\begin{proof}
  This follows from \cite[Corollary 5.30]{BV85:unipotent}, see also \cite[Proposition 3.33]{BMSZ:counting}.
\end{proof}

\begin{remark}
	When $\g$ is classical, we have
	\[
	\sigma_\nu \simeq \left(j_{W_\nu}^{W(\Lambda)} \sgn \right)\otimes \sgn,
	\]
	since the duality map tensoring with $\sgn$ preserves the set of special representations.
\end{remark}

\begin{proposition} \label{prop:equiv_LR}
	Let $\nu \in \Lambda = \lambda + X^*$ and $\sigma$ be a special representation of $W(\Lambda)$. 
	The following conditions are equivalent:
	\begin{enumerate}[label=(\roman*), nosep]
		\item $\sigma \leqLR \sigma_\nu$;
		\item $\check\sigma_\nu \leqLR \check\sigma$;
		\item 
			$\spr_{\ckfg_\Lambda}(\check{\sigma}) \preceq \Ind_{\ckfg_\nu}^{\ckfg_\Lambda} \bfzero$.
	\end{enumerate} 
\end{proposition}

\begin{proof}
	The equivalence between (i) and (ii) is by the duality \cite[Proposition 3.24]{BV85:unipotent}.
	By \cite[Proposition 3.23]{BV85:unipotent}
	
	Part (ii) is equivalent to closure relation $\spr_{\ckfg_\Lambda}(\check{\sigma}) \preceq \spr_{\ckfg_\Lambda}(\check{\sigma}_\nu)$ between the special nilpotent orbits in the dual Lie algebra $\ckfg$ whose corresponding Springer representations are $\check{\sigma}$ and $\check{\sigma}_\nu$. By \Cref{lem:j-ind_Richardson} and \Cref{prop:special_j}, we have
	\[\spr_{\ckfg_\Lambda}(\check{\sigma}_\nu) = \Ind_{\ckfg_\nu}^{\ckfg_\Lambda} \bfzero.\]
	This shows the equivalence between (ii) and (iii).

\end{proof}

\begin{corollary}\label{cor:closure}
	Let $\lambda \in \fh^*_\R$. Then the maximal primitive ideal $J_{max}(\lambda)$ is mildly unipotent if and only if
	\[
		\Ind_{\ckfg_\lambda}^{\ckfg_\Lambda} \bfzero  \preceq \Ind_{\ckfg_\nu}^{\ckfg_\Lambda} \bfzero  
	\]
    implies  $\lVert \nu \rVert \geq  \lVert \lambda \rVert$ for $\nu \in \Lambda = \lambda + X^*$.
\end{corollary}

\begin{proof}
    As in the proof of \Cref{prop:equiv_LR}, we have $\spr_{\ckfg_\Lambda}(\check{\sigma}_\lambda) = \Ind_{\ckfg_\lambda}^{\ckfg_\Lambda}\bfzero$. Now we appeal to the equivalence of the conditions (i) and (iii) in \Cref{prop:equiv_LR} in \Cref{defn:mildly_unip}.
\end{proof}

\subsection{Special unipotent ideals} \label{subsec:special_unipotent}

Let $\g$ be a general semisimple complex Lie algebra. It is shown in \cite[Proposition 5.10]{BV85:unipotent} that, when $\check{\Orb}$ is an even nilpotent orbit of $\check{\g}$, then the special unipotent ideal attached to $\check{\Orb}$ is weakly unipotent. Below we provide a slightly different proof of this result. First, we recall the following important result from \cite{BV85:unipotent}, which is also a key ingredient in the proof of \Cref{thm:q-unipotent_weaklyunip}.

\begin{proposition}\label{prop:norm_comparison}
	Let $\fg$ be any semisimple Lie algebra over $\C$ and $\nu$ be a semisimple element of $\ckfg$ that is integral (i.e., $\exp[2\pi i \ad(\nu)] = \Id_{\ckfg}$). Let $\ckfp = \ckfg_\nu\oplus \ckfu$ be the parabolic subalgebra defined by $\nu$. Suppose $\check{\Orb}$ is a nilpotent orbit in $\ckfg$ such that 
	\[ \check{\Orb} \preceq \Ind^{\ckfg}_{\ckfg_\nu} \bfzero, \]
	that is, $\check{\Orb}$ contains an element $\check{e}$ such that $\check{e} \in \fu$. Let $h$ be the semisimple element of an $\sl_2$-triple of $\check{\Orb}$ and $\lambda_{\check{\Orb}} = \frac{1}{2} h$. Then 
	\[ \lVert \nu\rVert \geq \lVert \lambda_{\check{\Orb}} \rVert,\]
	with equality holding if and only if $\nu$ is conjugate to $\lambda_{\check{\Orb}}$. In this case, $\check{\Orb}$ must be an even orbit and $\ckfp$ is the Jacobson-Morozov parabolic subalgebra attached to $e$ (cf. \cite[Corollary 5.6]{BV85:unipotent}) 
\end{proposition}

\begin{proof}
	This follows from Lemma 5.7 and the last part of the proof of Proposition 5.10 in \cite{BV85:unipotent}.
\end{proof}

\begin{theorem}[{\cite[Prop. 5.10]{BV85:unipotent}}] \label{thm:special_unipotent}
	All special unipotent ideals attached to even $\check{\Orb}$ are mildly unipotent with respect to the weight lattice. In particular, they are weakly unipotent. 
\end{theorem}

\begin{proof}
	This follows from $\|b\cdot \nu\| = \|\nu\|$ (for $b\in B$),  \Cref{cor:closure} (where $\ckfg_\Lambda = \ckfg$) and \Cref{prop:norm_comparison}.
\end{proof}

\begin{remark}
	Both the original proof of \cite[Proposition 5.10]{BV85:unipotent} and our proof use \Cref{prop:norm_comparison} in the last parts. See \Cref{rem:McGovern} for the difference.
\end{remark}

\section{The case of classical groups} \label{sec:classical}

\subsection{Notations on partitions}\label{partition_notation}

In classical groups it will be helpful  
to have a description of 
the elements of $\nilcone_o$ 
and the map $d$
in terms of partitions.  We introduce that notation
following the references \cite{C-M},  \cite{Carter}, \cite{Spaltenstein}.

Let $\parti (N)$ denote the set of partitions of $N$.
For $\bfd \in \parti(N)$,  we write $\bfd = [d_1, \dots, d_k]$, 
where $d_1 \geq \dots \geq d_k > 0$ and $|\bfd|:=  \sum_{j=1}^k d_j$ is equal to $N$. Let $\#\bfd$ denote the number of members of $\bfd$ (counting multiplicities).
Let
$m_\bfd(s)= \# \{j \ | \ d_j = s \}$
denote the multiplicity of the part $s$ in $\bfd$.  We use $m(s)$ if the partition is clear. If $s_1 > s_2 > \cdots > s_d$ are all distinct parts of $\bfd$, we also write $\bfd = [s_1^{m(s_1)}, s_2^{m(s_2)} \dots, s_d^{m(s_d)}]$.
The set of nilpotent orbits $\nilcone_{o}$ in $\g=\mathfrak{sl}_{n}$ under the 
adjoint action of $G=\SL_n$  is in bijection with 
$\parti (n)$. 

The \emph{height} of a part $s$ in $\bfd$ is the number
\[ h_\bfd(s):=\#\{ d_j \, | \, d_j \geq s \}.\]
We will also write $h(s)$ if the partition is clear. The \emph{transpose partition} of $\bfd \in \parti(N)$ is the partition $\bfd^t = [d^t_1 \geq \cdots d^t_k > 0] \in \parti(N)$ defined by 
\[ d^t_i = \# \{ j \ | \ d_j \geq i \}. \]

For $\epsilon \in \{ 0,1 \}$,
let $V=V_\epsilon$ be a vector space (over $\C$) of dimension $N$, equipped with a nondegenerate bilinear form satisfying 
$\langle v, v' \rangle = (-1)^\epsilon \langle v', v \rangle$ for $v,v' \in V$. 
Let $G=G_\epsilon(V) \subset \SL(V)$ be the classical Lie group consisting of linear automorphisms of $V$ fixing the bilinear form, so that $\g_\epsilon(V) = \Or(N) = \Or(N)$ is the (disconnected) orthogonal group when $\epsilon =0$,
and $\g_\epsilon(V) = \Sp(V) = \Sp(N)$ when $\epsilon=1$ and $N$ is even. Let $\g_\epsilon(V)$ be the Lie algebra of $G_\epsilon(V)$.

Let $$\parti_{\epsilon} (N):= \{\bfd \in \parti(N)  \ | \   m(s) \equiv 0 \,\bmod 2 \text{ whenever }   s \equiv \epsilon  \,\bmod 2 \}.$$
A partition $\bfd \in \parti_\epsilon(N)$ will be referred as an $\epsilon$-partition. Then the set of nilpotent orbits $\nilcone_{o}$ in $\g_\epsilon(V)$ under the  group $G_\epsilon(V)$ is given by $\parti_{\epsilon} (N)$. We will also write $\parti_{C} (2n) = \parti_{1} (2n)$, $\parti_{B} (2n+1) = \parti_{0} (2n+1)$ and $\parti_{0} (2n)$ as  $\parti_{D} (2n) = \parti_{0} (2n)$.
For $\bfd \in \parti(N)$ or $\bfd \in \parti_\epsilon(N)$, we denote by $\Orb_\bfd$ the corresponding
nilpotent $G$-orbit in the Lie algebra $\g$. We will often not distinguish the set of $G$-orbits and the set of relevant partitions. 

 Note that $\parti_{\epsilon} (N)$ also parametrize nilpotent orbits under the identity component group $G_\epsilon(V)^\circ$, except that, when $\g$ is of type $D$, those partitions
with all even parts, called the \emph{very even partitions}, correspond to two different $\SO(V)$-orbits in $\nilcone_{o}$, called
the \emph{very even orbits} (in this case $4$ divides $N$). In what follows we will never have a need to seperate the very even $\SO(N)$ orbits, so we will not bother to introduce extra notation to distinguish between very even orbits.  


There is a partial order on $\parti(N)$ 
defined by the dominance relation on partitions,
\[\bfp \preceq \bfq \iff \sum_{i=1}^k p_i \leq \sum_{i = 1 }^k q_i, \qquad \forall k \geq 1.\]
This restricts to partial orders on the subsets $\parti_\epsilon(N)$ and 
and these partial orders coincide
with the partial orders on the sets of nilpotent orbits given
by the closure relation.  
interchangeably in the classical groups (with the
caveat mentioned earlier for the very even orbits in type $D$).

Let $X=B$, $C$, or $D$.
Let $N$ be even (resp. odd) if $X$ is of type $C$ or $D$ (resp. $B$).
The {\it $X$-collapse} of $\bfd \in \parti(N)$ is defined as the
unique maximal partition $\bfd_X \in \parti_{X}(N)$ dominated by $\bfd$, i.e., we have
$\bfd_X \preceq \bfd$ and 
if $\mu \in \parti_{X}(N)$ and $\mu \preceq \bfd$, then $\mu \preceq \bfd_X$.
The $X$-collapse always exists and is unique. 

\subsection{Special partitions and the duality maps} \label{subsec:special_partitions}

All nilpotent orbits of type $A$ are special.  To describe the special nilpotent orbits in other classical Lie algebras and the extra notion of metaplectic special orbits in terms of partitions, we define four sets of partitions, with $\epsilon' \in \{0,1\}$, as follows:
\begin{equation}\label{eq:special_condition}
	\parti_{\epsilon,\epsilon'} (N) := \{\lambda \in \parti_{\epsilon}(N)  \ | \   h_\bfd(s) \equiv \epsilon' \,\bmod 2  \text{ whenever }  s \equiv \epsilon \,\bmod 2  \}.
\end{equation}
Note that, when $N$ is odd, the set is nonempty only when $(\epsilon,\epsilon') = (0,1)$ because of the $s=0$ case.
For $N$ even, the set is nonempty for $(\epsilon,\epsilon') = (0,0), (1,0)$ and $(1,1)$.
Then the partitions for the special orbits in Lie algebras of type $B_n,C_n$ and $D_n$ are given by 
$\parti^{sp}_B (2n+1):= \parti_{0,1} (2n+1)$, $\parti^{sp}_C (2n):=  \parti_{1,0} (2n)$
and $\parti^{sp}_D (2n):= \parti_{0,0} (2n)$.
The case of $(\epsilon,\epsilon') = (1, 1)$ leads to a second subset $\parti^{ms}_C (2n):= \parti_{1,1} (2n)$ of $\parti_C (2n)$.  We refer to the corresponding nilpotent orbits in type $C$ as the \emph{metaplectic special} nilpotent orbits. These four sets inherit the partial order from the set of all partitions, which agrees with the closure order of the corresponding nilpotent orbits. 

\begin{remark}
	The notion of metaplectic special orbits appeared earlier in \cite{Moeglin} (where it is called anti-special orbits),  \cite{Jiang} and \cite{BMSZ:metaplecticBV}. Note that \cite[Definition 1.1]{BMSZ:metaplecticBV} defines metaplectic special orbits as those corresponding to partitions of type $C$ whose transpose is of type $D$, which can easily be seen as equivalent to that of $\parti_{1,1} (2n)$.
	Here we are mostly following the definitions and notations from \cite[Section 2.2]{JLS:Duality}, except that metaplectic special orbits are referred to as \emph{alternative special orbits} and are denoted as $\parti^{asp}_C(2n)$ there.
\end{remark}

We will see here that there are order-preserving bijections between the sets $\parti^{sp}_B (2n+1)$ and $\parti^{sp}_C (2n)$ 
(see \cite{Spaltenstein}, \cite[Proposition 4.3]{Kraft-Procesi:special}), and between the sets $\parti^{sp}_D (2n)$ and $\parti^{ms}_C (2n)$. Given $\bfd =[d_1 \geq \dots \geq d_{k-1} \geq d_k > 0]$,
define the partitions 
\[\bfd^{-} = [d_1, \ldots, d_{k-1}, d_k -1 ]\]
and
\[\bfd^{+} = [d_1 +1, d_2, \dots, d_{k-1}, d_k].\]
Then the bijections are given as follows, denoted as $f_{XY}$:
\begin{equation}\label{eq:bijection_f}
	\begin{aligned}
		& f_{BC}: \parti^{sp}_B (2n+1) \to \parti^{sp}_C (2n), \quad f(\bfd) = (\bfd^{-}) _C \\
		& f_{CB}: \parti^{sp}_C (2n) \to \parti^{sp}_B(2n+1) , \quad  f(\bfd) = (\bfd^{+}) _B \\
		& f_{DC}: \parti^{sp}_D(2n) \to \parti^{ms}_C (2n) , \quad  f(\bfd) = ((\bfd^{+})^-) _C \\
		& f_{CD}: \parti^{ms}_C (2n) \to \parti^{sp}_D(2n) , \quad  f(\bfd) = \bfd _D  \\
	\end{aligned} 
\end{equation}
Note that in general $f_{XY}$ maps $\parti_{\epsilon,\epsilon'}$ to $\parti_{1-\epsilon,1-\epsilon'}$. Moreover, $f_{XY}$ and $f_{YX}$ are inverse to each other. The pair of bijections $f_{BC}$ and $f_{CB}$ is also sometimes referred as the Springer duality map. All these bijections are order-preserving. 


We just remark that the Lusztig-Spaltenstein duality map $d_{LS}$ and the Barbasch-Vogan duality map $d_{BV}$ in \Cref{subsec:special_duality} are given by $d_{LS} (\bfd) = (\bfd^t)$ and $d_{BV} = d_{LS} \circ f = f \circ d_{LS}$ respectively. Also there is a metaplectic BV duality. See \cite[Section 2.2]{JLS:Duality} and \cite{BMSZ:metaplecticBV} for details. We will not need these in the rest of the paper.


We identify the Cartan subalgebra $\fh$ of $\g = \so(2n+1) \subset \sl(2n+1)$ with $\C^n$, with the root system
\begin{equation}\label{eq:roots_B}
	\Phi:=\{e_i\pm e_j\mid 1\leq i<j\leq n\}\sqcup \{ \pm e_i\mid 1\leq i\leq n\}\subset \fh^*.
\end{equation}
and the positive root system
\begin{equation}\label{eq:posroots_B}
	\Phi^+:=\{\pm e_i \pm e_j\mid 1\leq i<j\leq n\}\sqcup \{ e_i\mid 1\leq i\leq n\}\subset \fh^*.
\end{equation}
Here $e_1, e_2,\cdots, e_n$ is the standard basis of $\C^n$, and we
identify the dual space $\fh^* = (\C^n)^*$ with $\fh = \C^n$ so that the basis $e_i$ are self-dual. The Weyl group $W_n\subset \GL_n(\C)$ of $\g$ is generated by
all the permutation matrices and the diagonal matrices of order
$2$. The identification $\fh \simeq \fh^*$ is $W$-equivariant. The Langlands dual Lie algebra $\ckfg = \sp(2n)$ of $\g$ has the universal Cartan subalgebra $\fh^*$, which is identified with $\fh \simeq \C^n$ as above, with the dual root system
\begin{equation}\label{eq:roots_C}
	\check{\Phi}:=\{\pm e_i \pm e_j\mid 1\leq i<j\leq n\}\sqcup \{ \pm 2e_i\mid 1\leq i\leq n\}\subset \fh^*.
\end{equation}
and the dual positive root system
\begin{equation}\label{eq:posroots_C}
	\check{\Phi}^+:=\{e_i\pm e_j\mid 1\leq i<j\leq n\}\sqcup \{ 2e_i\mid 1\leq i\leq n\}\subset \ckfh^* \simeq \fh.
\end{equation}
The Weyl group is also $W_n$. We identify $\so(2n)$ as a maximal pseudo-levi subalgebra $\g'$ of $\so(2n+1)$, so that $\fh$ can also be identified as the universal Cartan subalgebra of $\so(2n)$ with root systems 
\begin{equation}\label{eq:roots_D}
	\Phi' = \{\pm e_i\pm e_j\mid 1\leq i<j\leq n\}
\end{equation} 
and the positive root system
\begin{equation}\label{eq:posroots_D}
	\Phi'^+ = \{e_i\pm e_j\mid 1\leq i<j\leq n\}
\end{equation} 
The isomorphism $\fh \simeq \fh^*$ identifies $\Phi'$ (resp. $\Phi'^+$) with its dual (resp. positive) root system $\check{\Phi}'$ (resp. $\check{\Phi}'^+$), both are regarded as root subsystems of $\Phi$ and $\Phi^\vee$. We have inclusions of root systems $\Phi' \subset \Phi$, $\Phi'^+ \subset \Phi^+$, etc. Let $W'_n$ be the Weyl group of $\Phi'$ of type $D_n$, identified as a normal subroup of $W_n$ in the standard way. Then $W_n$ acts on $W_n'$ by conjugation and hence on $\Irr W_n'$.

We have the following alternative description of the bijections in \eqref{eq:bijection_f} in terms of (metaplectic) special representations of Weyl groups and Springer correspondence. We refer the reader to \cite{BMSZ:metaplecticBV} for the definition of metaplectic special representations of $W_n$, denoted as $\Irr^{ms}(W_n)$.

\begin{proposition} \label{prop:f_BC&f_DC}
	We have the following commutative diagrams:
	\begin{equation}\label{diag:f_CB}
		\begin{tikzcd}
			\parti_C^{sp}(2n) \ar[d,"\spr^{-1}"] \ar[r,"f_{CB}"] & \parti^{sp}_B (2n+1) \ar[d,"\spr^{-1}"]{}\\
			\Irr^{sp}(W_n) \ar[equal]{r} & \Irr^{sp}(W_n)
		\end{tikzcd}
	\end{equation}
	and 
	\begin{equation}\label{diag:f_DC}
		\begin{tikzcd}
			\parti_D^{sp}(2n) \ar[d,"\spr^{-1}"] \ar[r,"f_{DC}"] & \parti^{ms}_C (2n) \ar[d,"\spr^{-1}"]\\
			\Irr^{sp}(W'_n)/W_n \ar[r,"j_{W'_n}^{W_n}"]& \Irr^{ms}(W_n)
		\end{tikzcd}
	\end{equation}
	where the left vertical bijection in \eqref{diag:f_DC} is the map induced by the Springer correspondence of $\so(2n)$ and all the other vertical bijections are the usual Springer correspondences. The two top bijections are order-preserving. 
\end{proposition}

\begin{proof}
	\eqref{diag:f_CB} can be found in \cite[Chapitre III]{Spaltenstein}
	and \cite[Proposition 4.3]{Kraft-Procesi:special}. The compatibility of the Springer correspondence follows from easy computations using \cite{Carter}. \eqref{diag:f_DC} is Proposition 6.4 and Corollary 6.5 of \cite{BMSZ:metaplecticBV}.
\end{proof}

\begin{proposition}\label{prop:bijection_richardson}
	With the notations above, the bijections in \eqref{eq:bijection_f} commutes with taking Richardson orbits. More precisely:
	\begin{enumerate}
		\item 
			Let $\fm$ be a (standard) levi subalgebra in $\g = \so(2n+1)$ and $\ckfm$ its dual levi subalgebra in $\ckfg = \sp(2n)$. Then
			\begin{equation}  \label{eq:Ind_BC}
				\Ind^{\ckfg}_{\ckfm} \bfzero
			= f_{BC} \left( \Ind^{\fg}_{\fm} \bfzero \right). 
			\end{equation}
	 	\item 
	 	 	Let $\fm$ be a (standard) levi subalgebra in $\g=\so(2n+1)$ which is also contained in $\g'=\so(2n)$, and $\ckfm$ its dual levi subalgebra in $\ckfg=\sp(2n)$. Then
	 	 	\begin{equation}  \label{eq:Ind_DC}
	 	 		\Ind^{\ckfg}_{\ckfm} \bfzero
	 	 	= f_{DC} \left( \Ind^{\fg}_{\fm} \bfzero \right). 
	 	 \end{equation}
	\end{enumerate}
\end{proposition}

\begin{remark}
	Part (1) of Proposition \ref{prop:bijection_richardson} has also appeared in \cite[Theorem 1.3, Proposition 3.1]{Fu-Ruan-Wen}. Here we will give a more direct proof using $j$-induction, which also works for Part (2). On the other hand, \cite[Theorem 1.3]{Fu-Ruan-Wen} also discusses the duality between the coverings of the Richardson orbits in question induced by the generalized Springer maps. It is natural to expect the analogue also holds for Part (2), which we will not explore here.
\end{remark}

\begin{proof}
	Let $W_0 \subset W_n$ be the Weyl group of $\fm$, which is also the Weyl group of $\ckfm$. The Springer representation associated to the Richardson orbit $\Ind^{\fg}_{\fm} \bfzero$ in $\g^* = \so(2n+1)^*$ is the $W_n$ representation $j_{W_0}^{W_n} \sgn$. 
	
	Similarly, The Springer representation associated to the Richardson orbit $\Ind^{\ckfg}_{\ckfm} \bfzero$ in $\ckfg^* = \sp(2n)^*$ is also the $W_n$ representation $j_{W_0}^{W_n} \sgn$. Then (1) follows from Proposition \ref{prop:f_BC&f_DC}, (1).
	
	Part (2) follows from a similar argument to the above and Proposition \ref{prop:f_BC&f_DC}, (2).
	
\end{proof}

\begin{corollary} \label{cor:richardson}
	With the notations in \Cref{prop:bijection_richardson}, let $\lambda \in \fh$ and $\check{\lambda} \in \fh^*=\ckfh$ be the image of $\lambda$ under the isomorphism $\fh \simeq \fh^*$. Let $\fm=\g_{\lambda}$. Then its Langlands dual levi subalgebra in $\ckfg$ is $\ckfm = \ckfg_{\check{\lambda}}$. Then we have \eqref{eq:Ind_BC}. If in addition $\g_{\lambda} \subset \g'$, then we also have \eqref{eq:Ind_DC}.
\end{corollary}

\subsection{\texorpdfstring{$q$-unipotent ideals}{q-unipotent ideals}}

We first recall the notion the \emph{$q$-unipotent} infinitesimal characters from \cite[\S\,4]{McGovern1994}.

\begin{definition}[{\cite[Definition 4.11]{McGovern1994}}] \label{defn:q-unipotent}
	Let $\g$ be a simple Lie algebra of classical type and $\ckfg$ the Langlands dual of $\g$. Let $\check{V}$ be the standard representation of $\ckfg$, then we have natural inclusion $\ckfg \hookrightarrow \sl(\check{V})$. When $\ckfg$ is of type $C_n$, we consider instead the composition of inclusions $\ckfg=\sp(2n) \hookrightarrow \sl(\check{V}) = \sl(2n) \hookrightarrow \sl(2n+1)$, where $\sl(2n) \hookrightarrow \sl(2n+1)$ is the standard embedding as block diagonal matrices (with $0$ in, say, the lower right corner). With this convention, we have fixed $\ckfg \hookrightarrow \sl(N')$ with $N' = N$ (resp. $N' = 2n+1$) when $\ckfg = \sl(N)$ or $\so(N)$ (resp. $\sp(2n)$).

	Let $\check{\Orb}$ be a nilpotent orbit in $\sl(N')$ corresponding to a partition $\bfq$ of $N'$. Let $\lambda_{\check{\Orb}} = \half h_{\check{\Orb}}$ be as in \Cref{subsec:special_unipotent}. Let $\lambda'_{\check{\Orb}}$ be any $\SL(N')$-conjugate of $\lambda_{\check{\Orb}}$ that lies in a fixed Cartan subalgebra $\ckfh$ of $\ckfg$. Then we can regard $\lambda'_{\check{\Orb}}$ as an infinitesimal character for $\fg$, which is called a \emph{$q$-unipotent} infinitesimal character for $\g$. The maximal primitive ideal $J_{max}(\lambda'_{\check{\Orb}})$ is called a \emph{$q$-unipotent ideal} of $\Ug$. 
\end{definition}

\begin{theorem}[{\cite[Theorem 4.10]{McGovern1994}}] \label{thm:defn_q-unipotent}
	The map $\check{\Orb} \mapsto \lambda'_{\check{\Orb}}$ from the set of nilpotent orbits in $\sl(N')$ to the set $\ckfh/W$ of infinitesimal characters for $\g$ is well-defined in all cases, except up to an outer automorphism in type $D$. More precisely, if $\fg$ is not of type $D$, or if $\fg$ is of type $D$ and the partition $\bfq$ of $\check{\Orb}$ has at least one odd term, then any choices of $\lambda'_{\check{\Orb}}$ are conjugate under the action of the Weyl group $W$. If $\fg$ is of type $D$ and the partition $\bfq$ of $\check{\Orb}$ has only even terms, then there are two choices of $\lambda'_{\check{\Orb}}$ up to $W$-conjugacy, each differing from each other by an outer automorphism of $\ckfg$.
\end{theorem}

\begin{remark}
	We will see in the last part of the proof of \Cref{thm:q-unipotent_weaklyunip} that why it is natural to consider the exceptional inclusion $\sp(2n) \hookrightarrow \sl(2n+1)$ in \Cref{defn:q-unipotent}.
\end{remark}

\begin{remark}
	Note that when $\g$ is of type $A$, the set of $q$-unipotent infinitesimal characters coincide with the special unipotent infinitesimal characters defined in \Cref{subsec:special_unipotent}. Therefore we will assume that $\g$ is not of type $A$ for the rest of this subsection.
\end{remark}

\begin{definition}[{\cite[Definition 8.2.1]{LMBM}}] \label{defn:rho+}
	Suppose $\bfq = [q_1, q_2, \ldots, q_l]$ is a partition of $N$. Define $\rho^+(\bfq) \in \left(\frac{1}{2} \Z \right)^{\lfloor \frac{N}{2} \rfloor}$ by appending the \emph{positive} elements of the sequence
	\[ \left(  \frac{q_i - 1}{2}, \frac{q_i - 3}{2}, \ldots, \frac{3 - q_i}{2}, \frac{1 - q_i}{2}  \right) \]
	for each $i \geqslant 1$, and then adding $0$'s if necessary so that the length of the sequence $\rho^+(\bfq)$ equals $\lfloor \frac{N}{2} \rfloor$.
\end{definition}

For classical $\fg$ not of type $A$, one can express the infinitesimal character $\lambda'_{\check{\Orb}}$ attached to orbit $\check{\Orb}$ with partition $\bfq$ in terms of the standard coordinates as in \Cref{subsec:special_partitions}. When $\fg$ is not of type $D$, or when $\fg$ is of type $D$ and the partition $\bfq$ has at least one odd term, $\lambda'_{\check{\Orb}}$ is given by $\rho^+(\bfq)$ up to $W$-conjugacy. Note that in this case, $\rho^+(\bfq)$ always has at least one zero coordinate. When $\fg$ is of type $D$ and the partition $\bfq$ has only even terms, then the two choices of $\lambda'_{\check{\Orb}}$ up to $W$-conjugacy are given by $\rho^+(\bfq)$, whose coordinates are all non-zero, and the other choice given by multiplying the last coordinate of $\rho^+(\bfq)$ by $-1$.

\begin{remark}\label{rem:McGovern}
	In \cite[Definition 5.5]{McGovern1994}, McGovern defined the notion of \emph{parabolically weak unipotence} of a (Harish-Chandra) module of $\Ug$. A parabolically weakly unipotent module is in particular weakly unipotent with respect to the root lattice in the sense of Definition \ref{defn:weakly_unipotent}. McGovern claimed in \cite[Theorem 5.6]{McGovern1994} that for any $q$-unipotent infinitesimal character $\lambda \in Q$, he could prove that the associated $q$-unipotent ideal $J_{max}(\lambda)$ is parabolically weakly unipotent by checking the following sufficient (but not necessary) condition: whenever $\g$ is isomorphic
	to the derived algebra $[\fl, \fl]$ of a Levi factor of a larger semisimple algebra $\g'$, any $\lambda'$ congruent to $\lambda$ modulo the root lattice of $\g'$ satisfies one of the following conditions:
	\begin{enumerate}
		\item the associated variety $\cV(J_{max}(\lambda'))$ of $J_{max}(\lambda')$ is not contained in $\cV(J_{max}(\lambda))$\footnote{In the original proof of \cite[Theorem 5.6]{McGovern1994}, this condition was mistakenly stated as ``the associated variety $\cV(J_{max}(\lambda'))$ of $J_{max}(\lambda')$ differs from $\cV(J_{max}(\lambda))$".}, or
		
		\item $\lVert \lambda' \rVert \geq \lVert \lambda \rVert$.
	\end{enumerate}
	This is exactly how Barbasch and Vogan proved in \cite[Proposition 5.10]{BV85:unipotent} that the special unipotent ideal of $\Ug$ attached to an even nilpotent orbit $\check{\Orb}$ in $\check{\g}^*$ is weakly unipotent (with respect to the weight lattice). In \Cref{ex:counterexample1} below, however, we will see that there exist $\lambda, \lambda' \in Q$ that are congruent modulo the root lattice of $\g$ itself, such that neither of the two conditions above is satisfied. Therefore one cannot prove parabolically weak unipotence, or even weak unipotence, of $q$-unipotent ideals via the approach proposed by McGovern above.
	
\end{remark}

\begin{example}\label{ex:counterexample1}
	Let $\ckfg = \so(20)$ and set 
	  \[ \bfp_1 = [9,1] \in \parti_{D}^{sp}(10), \quad \bfp_2 = [5,5] \in \parti_{D}^{sp}(10).\] 
	Let 
		\[\bfp'_1 = f_{DC}(\bfp_1) = [10], \quad \bfp'_2 = f_{DC}(\bfp_2) = [6, 4], \] 
	both of which belong to $\parti_{C}^{ms}(10)$ (cf. the proof of \Cref{thm:q-unipotent_weaklyunip}). 
	
	Let $\Orb_1 := \Orb_{\bfp'_1 \cup \bfp_2}$, $\Orb_2: = \Orb_{\bfp'_2 \cup \bfp_1}$ be nilpotent orbits in $\sl(20)$. We take 
	\[ \lambda_1 = \lambda_{\Orb_1} = \tfrac{1}{2} h_{\Orb_1} = \rho^+ (\bfp'_1 \cup \bfp_2) = \left( \frac{9}{2}, \frac{7}{2}, \frac{5}{2}, \frac{3}{2}, \frac{1}{2}, 2, 1, 2, 1, 0 \right)\]
	and
	\[ \lambda_2 = \lambda_{\Orb_2} = \frac{1}{2} h_{\Orb_2} = \rho^+ (\bfp'_2 \cup \bfp_1) = \left( \frac{5}{2}, \frac{3}{2}, \frac{1}{2}, \frac{3}{2}, \frac{1}{2}, 4, 3, 2, 1, 0 \right).\]
	It is clear that $\lambda_1, \lambda_2 \in Q$.
	Then $\lambda_1$ and $\lambda_2$ differ by an element in the coroot lattice of root lattice of $\g = \so(20)$ (since the difference of the sums of coordinates of $\lambda_1$ and $\lambda_2$ is even) and $\lVert \lambda_1 \rVert > \lVert \lambda_2 \rVert$. However, one can compute that $J_{max}(\lambda_1)$ and $J_{max}(\lambda_2)$ have the same associated variety in $\g^* = \so(20)^*$, which is the closure of the orbit with partition $[3^5, 1^5]$ (e.g., by \cite{BMW:annihilator_var}). Therefore both conditions in Remark \ref{rem:McGovern} fail.

	Here is a natural reason for $\cV(J_{max}(\lambda_1)) = \cV(J_{max}(\lambda_2))$: The pseudo-levi subalgebra of $\g$ corresponding to the integral root system of $\lambda_1$ is the same as that of $\lambda_2$, which is $\so(10) \times \so(10)$. Let $\Orb_i$ be the nilpotent orbit in $\g=\so(20)$ such that $\overline{\Orb}_i = \cV(J_{max}(\lambda_i))$, $i=1,2$. 
	Then the Springer representations corresponding to $\Orb_1$ and $\Orb_2$ are of the form 
		\[ j_{W'_{5} \times W'_{5}}^{W'_{10}} \sigma_1 \otimes \sigma_2 \quad \text{and} \quad j_{W'_{5} \times W'_{5}}^{W'_{10}} \sigma_2 \otimes \sigma_1, \]
	respectively, where $\sigma_i$ are special representations of $W'_{5}$, the Weyl group of $\so(10)$, $i=1,2$. It is easy to see that these two representations are isomorphic, which implies $\Orb_1 = \Orb_2$.
\end{example}

We provide a correct proof of weak unipotence of $q$-unipotent ideals below.

\begin{theorem}\label{thm:q-unipotent_weaklyunip}
	All $q$-unipotent ideals are mildly and hence weakly unipotent with respect to the root lattice $Q_\g$ of $\g$.
\end{theorem}

\begin{proof}
	We treat the cases when $\ckfg$ is of type $B$ or $D$ in detail. The argument for $\ckfg$ of type $C$ is along the same lines and we only mention necessary changes.
	\vskip 0.5 em
	\noindent {\bf Types $B$ and $D$.} Suppose $\ckfg = \mathfrak{so}(N)$, where $N \geqslant 7$. Let $n = \lfloor \frac{N}{2} \rfloor$.
	Let $\check{\Orb}$ be any orbit in $\sl(N)$ whose corresponding partition is $\bfd$. Take an $\sl_2$-triple of $\check{\Orb} \subset \sl(N)$ with semisimple element $h_{\check{\Orb}} = h_\bfd \in \ckfh$. Let $\bfd_i$, $i = 0, 1$, be the subpartition of $\bfd$ consisting of all members of $\bfd$ that are congruent to $i$. Set $N_i = |\bfd_i|$, $i = 0, 1$. Then $N_0$ is even and $\# \bfd_1 \equiv N_1 \equiv N \bmod 2$.  
	
	The pseudo-levi subalgebra of $\ckfg$ corresponding to the integral root system of $\lambda_{\check{\Orb}} = \frac{1}{2} h_{\check{\Orb}}$ is $\fl = \fl^0 \oplus \fl^1$, where $\fl^0 = \so(N_0)$ and $\fl^1 = \so(N_1)$. The Cartan subalgebra $\ckfh$ of $\ckfg$ decomposes into the direct sum of the Cartan subalgebras $\ckfh_i$ of $\fl^i$, $i=0, 1$.

	Then we can regard $\bfd_1 \in \parti_{0}(N_1)$, which is of type $B$ (resp. $D$) when $\ckfg$ is of type $B$ (resp. $D$). Moreover, since $\bfd_1$ consists of only odd members, it is automatically a special partition and hence belongs to $\parti_{0,\epsilon'}(N_1)$, where $\epsilon' \equiv N_1 \bmod 2$. It determines a nilpotent orbit $\check{\Orb}_{\bfd_1}$ in $\fl^1$. This is ordinary (think about the case when $\bfd_0 = \varnothing$, i.e., $\bfd = \bfd_1$, which corresponds to an even orbit in $\ckfg$). 
	
	On the other hand, $\bfd_0$ consists of only even members and each of its members might not necessarily have even multiplicity in general. Therefore we cannot regard $\bfd_0$ as a partition in $\parti_0(N_0)$. It turns out that the right thing to do here is to regard $\bfd_0$ as a metaplectic special partition in $\parti_{1,1}(N_0) = \parti_C^{ms}(N_0)$. Let $\fm = \sp(N_0)$, which has a Cartan subalgebra identified with the Cartan subalgebra $\ckfh_0$ of $\fl^0 = \so(N_0)$ as mentioned above, so that $\bfd_0$ corresponds to a nilpotent orbit $\check{\Orb}_0 := \check{\Orb}_{\bfd_0}$ in $\fm$.
	
	Let $h_i := h_{\bfd_i}$ be the semisimple element of an $\sl_2$-triple of $\check{\Orb}_{\bfd_i}$, $i=0, 1$, regarded as nilpotent orbit in the Lie algebras as above. Then $\lambda_{\check{\Orb}} = \frac{1}{2} h_\bfd = \frac{1}{2}h_{\check{\Orb}}$ can be regarded as the concatenation of $\lambda_0 = \frac{1}{2}h_0 $ and $\lambda_1 = \frac{1}{2}h_1$ (up to conjugation by the Weyl group $S_{m-1}$ of $\sl(m)$). In terms of the standard coordinates $\{e_i\}$, $\lambda_0 \in (\frac{1}{2} + \Z)^{N_0}$ and $\lambda_1 \in \Z^{N_1}$. 
	
	Now given any $\gamma \in \lambda_{\check{\Orb}} + Q_\g$, the integral root system of $\gamma$ is the same as that of $ \lambda_{\check{\Orb}}$. We write $\gamma = \gamma_0 + \gamma_1$, such that $\gamma_i \in \ckfh_i$ for $i= 0, 1$. Then in terms of the standard coordinates $\{e_i\}$, we also have $\gamma_0 \in (\frac{1}{2} + \Z)^{N_0}$ and $\gamma_1 \in \Z^{N_1}$\footnote{Note that this is still true if we take the weight lattice $\bigoplus_{i=1}^n \Z e_i$ in place of the root lattice $Q_\g$ when $\g$ is of type $C$. However, this is no longer true if we take the weight lattice $\bigoplus_{i=1}^n \Z e_i + \Z (\frac{1}{2}\sum_{i=1}^n e_i) $ when $\g$ is of type $D$.}. Therefore $\lambda_0$ and $\gamma_0$ (resp. $\lambda_1$ and $\gamma_1$) are integral in $\fm$ (resp. $\fl^1$).
	
	Set $\fl_{\gamma} = \mathfrak{z}_{\ckfg} (\gamma) = \mathfrak{z}_{\fl} (\gamma) \subset \fl$ and $\fl^i_{\gamma_i} = \mathfrak{z}_{\fl^i} (\gamma_i) \subset \fl^i$, $i=0, 1$, which are levi subalgebras corresponding to singular datum associated to $\gamma$ and $\gamma_i$. Now assume
	 \begin{equation} \label{eq:closure1}
	 	\Ind^{\fl^i}_{\fl^i_{\lambda_i}} \bfzero \preceq \Ind^{\fl^i}_{\fl^i_{\gamma_i}} \bfzero 
	 \end{equation} 
	for $i=1, 2$. Here the Richardson orbits are understood as in $\fl^i$, $i= 0, 1$, respectively. 
	Since $\check{\Orb}_{\bfd_1}$ is an even orbit in $\fl^1$, we have $\Ind^{\fl^1}_{\fl^1_{\lambda_1}}\bfzero = \check{\Orb}_{\bfd_1}$ and this is a birational induction ({\bf cite}).
	By Proposition \ref{prop:norm_comparison}, we have $\lVert \lambda_1 \rVert \leq \lVert \gamma_1 \rVert$. For $\bfd_0$, apply $f_{DC}$ to the two sides of \eqref{eq:closure1} to get
	\begin{equation}\label{eq:closure2}
		\check{\Orb}_{\bfd_0}
	 	= \Ind^{\fm}_{\fm_{\lambda_0}} \bfzero
	 	= f_{DC} \left( \Ind^{\fl^0}_{\fl^0_{\lambda_0}} \bfzero \right) \preceq 
	 	 f_{DC} \left(  \Ind^{\fl^0}_{\fl^0_{\gamma_0}} \bfzero \right)
	 	 = \Ind^{\fm}_{\fm_{\gamma_0}} \bfzero.
	\end{equation}
	Here the second and the last equality is by Corollary \ref{cor:richardson}. Indeed, the last condition in \Cref{cor:richardson} is satisfied due to the observation that $\lambda_0, \gamma_0 \in (\frac{1}{2} + \Z)^{N_0}$\footnote{See \Cref{rem:root_vs_weight}}. The inequality in the middle is by the fact that $f_{DC}$ is order-preserving (\Cref{prop:f_BC&f_DC}). Now we can apply Proposition \ref{prop:norm_comparison} to $\fm$ and \eqref{eq:closure2}, and conclude that $\lVert \lambda_0 \rVert \leq \lVert \gamma_0 \rVert$. Therefore we have 
	 	\[ \lVert \lambda_{\check{\Orb}} \rVert^2 = \lVert \lambda_0 \rVert^2 + \lVert \lambda_1 \rVert^2 \leq \lVert \gamma_0 \rVert^2 + \lVert \gamma_1 \rVert^2 = \lVert \gamma \rVert^2.\] 
	By \Cref{cor:closure}, this finishes the proof in types $B$ and $D$.

	\vskip 0.5em 
	\noindent {\bf Type $C$.} Let $\ckfg = \sp(2n)$. Let $\check{\Orb}$ be any orbit in $\sl(2n+1)$ whose corresponding partition is $\bfd$ (recall from \Cref{defn:q-unipotent} that we consider $\ckfg=\sp(2n)$ as a subalgebra of $\sl(2n+1)$ in this case).

	Define $\bfd_i$, $N_i$, $h_i$, $i = 0, 1$, and so on as in the case of type $B$ and $D$, and we can argue in exactly the same way, except that the roles of $\bfd_0$ and $\bfd_1$ are interchanged. More precisely, regard $\bfd_0$ as an orbit $\check{\Orb}_{\bfd_0}$ in $\fl^0 \simeq \sp(N_0)$ and regard $\bfd_1$ as an orbit $\check{\Orb}_{\bfd_1}$ in $\fm := \so(N_1)$ of type $B$ (recall that $|\bfd| = 2n+1$ is odd, so $N_1$ is odd). Note that $\lambda_1, \gamma_1 \in \Z^{N_1}$, so that they are all integral regarded as elements in $\fm$. The analogue of \eqref{eq:closure2} is
	\begin{equation}
		\check{\Orb}_{\bfd_1}
		= \Ind^{\fm}_{\fm_{\lambda_1}} \bfzero
		= f_{CB} \left( \Ind^{\fl^1}_{\fl^1_{\lambda_1}} \bfzero \right) \preceq 
		f_{CB} \left( \Ind^{\fl^1}_{\fl^1_{\gamma_1}} \bfzero \right)
		= \Ind^{\fm}_{\fm_{\gamma_1}} \bfzero.
	\end{equation}
\end{proof}

\begin{remark} \label{rem:root_vs_weight}
	When $\ckfg$ is of $C$ ($\fg$ is of type $B$), the above proof also works if we replace the root lattice by the weight lattice 
		\[\bigoplus_{i=1}^n \Z e_i\]
	However, when $\ckfg$ is of type $B$ (resp. $D$), so that $\fg$ is of type $C$ (resp. $D$), the proof no longer works if we replace the root lattice $Q_\g$ by the weight lattice 
		\[\bigoplus_{i=1}^n \Z e_i + \Z\left(\frac{1}{2}\sum_{i=1}^n e_i \right)\] 
	and the paritition $\bfd$ of $\check{\Orb}$ has even parts, since in this case the condition $\gamma_0 \in (\frac{1}{2} + \Z)^{N_0}$ in the proof of \Cref{thm:q-unipotent_weaklyunip} might not hold any more and neither does the last condition in \Cref{cor:richardson}. See \Cref{ex:root_vs_weight} below.
	
	When $\bfd$ has only odd parts (so that $\check{\Orb}$ is an even orbit in $\ckfg$), however, the proof still works for weight lattices since $\gamma = \gamma_1$, which is just the case treated by \Cref{thm:special_unipotent}. 
\end{remark}

\begin{example} \label{ex:root_vs_weight}
	Let $\ckfg = \so(2n)$ and take $\check{\Orb}$ to be the regular nilpotent orbit in $\sl(2n)$, corresponding to the partition $[2n]$. Then 
		\[\lambda = \lambda_{\check{\Orb}} = \left(\frac{2n-1}{2}, \frac{2n-3}{2}, \ldots, \frac{1}{2}\right).\] 
	The integral root system of $\lambda_{\check{\Orb}}$ is of type $D_{n}$ and hence $\fl = \ckfg$.
	Take 
		\[\gamma = \lambda - \left(\frac{1}{2}, \frac{1}{2}, \ldots, \frac{1}{2} \right) = \left(n-1, n-3, \ldots, 1, 0 \right) = \rho_\g\] 
	to be the half sum of positive roots of $\g$. 
	Then $\lVert \gamma \rVert < \lVert \lambda_{\check{\Orb}} \rVert$ and $\ckfg_\gamma = \ckfg_\lambda = \ckfh$. Therefore $\Ind_{\ckfg_\gamma}^{\ckfg} \bfzero = \Ind_{\ckfg_\lambda}^{\ckfg} \bfzero$ is the regular nilpotent orbit in $\ckfg$. Therefore $\lambda$ does not satisfy the condition in \Cref{cor:closure} and hence is not mildly unipotent with respect to the weight lattice of $\g=\so(2n)$. Note that, if we set $\ckfm=\so(2n+1) \supset \ckfg = \so(2n)$ and $\fm=\sp(2n)$ with $\ckfh \subset \ckfm$ as in the proof of \Cref{thm:q-unipotent_weaklyunip} (see \Cref{prop:bijection_richardson}), then $\ckfm_\gamma \simeq \so(3) \times \gl(1)^{n-1}$ does not lie in $\ckfg$, therefore it does not satisfy the last condition in \Cref{cor:richardson} or the condition in \Cref{prop:bijection_richardson}, (2), hence the middle inequality in \eqref{eq:closure2} might not hold. Indeed, $\Ind_{\ckfm_\gamma}^{\ckfm} \bfzero$ is the subregular nilpotent orbit in $\ckfm$, which is strictly smaller than $\Ind_{\ckfm_\lambda}^{\ckfm} \bfzero$, the regular nilpotent orbit in $\ckfm$.

	This is not surprising, since the finite dimensional representation of $\g=\so(2n)$ of highest weight $\lambda - \rho_\g = (\frac{1}{2}, \ldots, \frac{1}{2})$ is $S_+$, one of the two half-spin representations of $\Spin(2n)$, and $S_+ \otimes S_+ \simeq S_+ \otimes S_+^* = \operatorname{End}(S_+)$ contains the trivial representation of $\g$, whose infinitesimal character is $\gamma=\rho_\g$. Therefore $J_{max}(\lambda)$ is not weakly unipotent with respect to the weight lattice of $\g=\so(2n)$. This example also appears below Theorem 5.6 of \cite{McGovern1994}. On the other hand, \Cref{thm:q-unipotent_weaklyunip} implies that $J_{max}(\lambda)$ is mildly and hence weakly unipotent with respect to the root lattice of $\so(2n)$.
\end{example}

\subsection{Metaplectic special unipotent ideals}

 Recall metaplectic special unipotent ideals from \cite{BMSZ:metaplecticBV}.

\begin{theorem} \label{thm:ms}
	Any metaplectic special unipotent ideal $J_{max}(\lambda_{\check{\Orb}})$ for $\fg = \sp(2n)$ is mildly and hence weakly unipotent with respect to the root lattice of $\g$.
\end{theorem}

\begin{proof}
	We observe that metaplectic special unipotent ideals are $q$-unipotent. Indeed, if $\check{\Orb}$ is a nilpotent orbit in the metaplectic dual $\fg^{ms} = \sp(2n)$ of $\fg$, corresponding to a partition $\bfq = [q_1, q_2, \ldots q_l] \in \parti_C(2n)$, then the partition
		\[ \bfq_+ := ((\bfq^t)^+)^t = [q_1, \dots, q_l, 1] \]
	defines a $q$-unipotent infinitesimal character $\lambda_{\bfq_+} = \frac{1}{2} h_{\check{\Orb}_+}$ (up to conjugation), where $\check{\Orb}_+$ is the nilpotent orbit in $\sl(2n+1)$ corresponding to $\bfq_+$. We have $\lambda_{\bfq_+} = \lambda_{\check{\Orb}}$. The rest follows from Theorem \ref{thm:q-unipotent_weaklyunip}.
\end{proof}

\subsection{The case of type \texorpdfstring{$A$}{A}} \label{subsec:type_A}

Suppose $\bfq$ is a partition of $n$. Write the lengths of the columns of the Young diagram of $\bfq$ as $c_1 \geq c_2 \geq \cdots c_m >0$. In other words, $\bfq^t  = [c_1, c_2, \dots, c_m]$. 

\begin{definition} \label{defn:xi_r}
	Suppose $\bfq = [q_1, q_2, \ldots, q_l]$ is a partition of $n$ with columns $\bfq^t  = [c_1 \geq c_2 \geq \dots \geq c_m]$.  Given $r = (-\half, \half]$, define an $n$-tuple $\xi_{r}(\bfq) \in \R^n$ by appending the sequences
	\begin{equation}\label{eq:rho_r_rows}
		\left( r -(-1)^{\lfloor r \rfloor} \left\lfloor \frac{q_i}{2} \right\rfloor, r -(-1)^{\lfloor r \rfloor} \left(\left\lfloor \frac{q_i}{2} \right\rfloor -1 \right),  \cdots, r, r+1,  \cdots,  r + (-1)^{\lfloor r \rfloor} \left\lfloor \frac{q_i - 1}{2}  \right\rfloor \right) 
	\end{equation}
	for each $i \geqslant 1$.
	Alternatively,  up to permutation, $\xi_{r}(\bfq) \in \R^n$ is the $n$-tuple
	\begin{multline}
	 	( 
	 	\underbrace{r, \ldots, r}_{c_1}, 
	 	\underbrace{r-(-1)^{\lfloor r \rfloor}, \ldots, r-(-1)^{\lfloor r \rfloor}}_{c_2}, 
	 	\underbrace{r + (-1)^{\lfloor r \rfloor},\ldots,r+(-1)^{\lfloor r \rfloor}}_{c_3}, \cdots, \\
		 \underbrace{r + (-1)^{j-1+\lfloor r \rfloor} \left\lfloor \frac{j}{2} \right\rfloor, \ldots,r + (-1)^{j-1+ \lfloor r \rfloor} \left\lfloor \frac{j}{2} \right\rfloor}_{c_j}, \cdots, \\
		 \underbrace{r + (-1)^{m-1 + \lfloor r \rfloor} \left\lfloor \frac{m}{2} \right\rfloor, \ldots,r + (-1)^{m-1 + \lfloor r \rfloor} \left\lfloor \frac{m}{2} \right\rfloor}_{c_m}
		).
	\end{multline}
	In other words, 
	\[ \xi_{r}(\bfq) = (\underbrace{r, \ldots, r}_{c_1}, \underbrace{r-1, \ldots, r-1}_{c_2}, \underbrace{r+1,\ldots,r+1}_{c_3}, \underbrace{r-2, \ldots, r-2}_{c_4}, \underbrace{r+2, \ldots, r+2}_{c_5}, \ldots) \]
	when $0 \leq r \leq \half$, and
	\[ \xi_{r}(\bfq) = (\underbrace{r, \ldots, r}_{c_1}, \underbrace{r+1, \ldots, r+1}_{c_2}, \underbrace{r-1,\ldots,r-1}_{c_3}, \underbrace{r+2, \ldots, r+2}_{c_4}, \underbrace{r-2, \ldots, r-2}_{c_5}, \ldots) \]
	when $-\half < r < 0$.

	We also generalize the above definition by replaceing $r$ by an $l$-tuple $\bfr = (r_1, r_2, \ldots, r_l)$ with each $r_i \in (-\half, \half]$. Define
	\[ \xi_{\bfr}(\bfq) := \xi_{r_1}([q_1]) \cup \xi_{r_2}([q_2]) \cup \cdots \cup \xi_{r_l}([q_l]) \in \R^n. \]
	In other words, $\xi_{\bfr}(\bfq)$ is obtained by appending the sequences \eqref{eq:rho_r_rows} with $r$ replaced by $r_i$, for all $1 \leq i \leq l$.
\end{definition}

Let $\lambda_{\check{\Orb}, r} = \xi_{r}(\bfq) \in \ckfh$ be the infinitesimal character defined above.

\begin{lemma} \label{lem:gl_min}
	Let $\ckfg = \gl(n)$ and $\check{\Orb} = \check{\Orb}_\bfq$ be a nilpotent orbit in $\ckfg$ corresponding to a partition $\bfq$ of $n$. Under the constraints that 
	\[ \Ind_{\ckfg_\nu}^{\ckfg} \bfzero = \check{\Orb}_\bfq, \]
	and that all coordinates of $\nu \in \ckfh$ belong to $r + \Z$, $\nu = \xi_{r}(\bfq)$ reaches the minimum of the norm $\lVert \nu \rVert$. 

	When $r \notin \{0, \half\}$, $\nu = \xi_{r}(\bfq)$ is the unique element in $\ckfh / W$ that reaches the minimum of the norm $\lVert \nu \rVert$ under the constraints above. 
	
	When $r = 0$ (resp. $\half$), the other elements that reach the minimum can be obtained from $\xi_{r}(\bfq)$ by replaceing all subsequences of the form \eqref{eq:rho_r_rows} corresponding to even (resp. odd) $q_i$ by their negatives (i.e., multiplying each coordinate of the sequences by $-1$).
\end{lemma}

\begin{proof}
	Write $\bfq^t  = [c_1 \geq c_2 \geq \dots \geq c_m]$. By \cite[Section 7.2]{C-M}, $\ckfg_\nu$ must be conjugate to the standard levi subalgebra $\mathfrak{gl}(c_1) \times \mathfrak{gl}(c_2) \times \cdots \times \mathfrak{gl}(c_l)$ corresponding to the subset
	\[ 
	  	\{e_1 - e_2, \ldots, e_{c_1 - 1} - e_{c_1}, e_{c_1 + 1} - e_{c_1 + 2}, \ldots, e_{c_1 + c_2 - 1} - e_{c_1 + c_2}, e_{c_1 + c_2 + 1} - e_{c_1 + c_2 + 2}, \ldots, e_{n-1} - e_n\} 
	\]
	of simple roots in $\Pi$. In particular, $\nu$ has exactly $c_1$ coordinates equal to some number $x_1 \in r + \Z$, exactly $c_2$ coordinates equal to some number $x_2 \in r + \Z$, etc. Therefore, up to permutation, $\nu$ must be of the form
	 \[
	 	( 
	 	\underbrace{x_1, \ldots, x_1}_{c_1}, 
	 	\underbrace{x_2, \ldots, x_2}_{c_2},  \cdots,
		 \underbrace{x_j, \ldots,x_j}_{c_j}, \cdots, 
		 \underbrace{x_m, \ldots,x_m}_{c_m}
		)
	 \]
	for some distinct $x_j \in r + \Z$, $1 \leq j \leq m$. 

	Set $a_j := x_j^2$ for all $j$. We have $\lVert \nu \rVert^2 = \sum_{j=1}^{m} c_j a_j$. It is easy to see that this sum only reaches the minimum when the sequence $a_j$ is non-decreasing (otherwise swapping $a_i$ and $a_j$ for some $i<j$ with $a_i > a_j$ will decrease the sum).
	Since $x_j \in r + \Z$, the minimum of $\lVert \nu \rVert$ is attained at 
		\[ (x_1, \ldots, x_m) = (r, r-1, r+1, r-2, r+2, \ldots) \]
	when $0 \leq r \leq \half$, and at
		\[ (x_1, \ldots, x_m) = (r, r+1, r-1, r+2, r-2, \ldots) \]
	when $-\half < r < 0$. In other words, $\lVert \nu \rVert$ reaches minimum under the constraints above at $\nu = \xi_{r}(\bfq)$. The remaining statements concerning other elements reaching the minimum are easy to verify.
\end{proof}

\begin{lemma} \label{lem:gl_compare}
	Let $\bfq_1$ and $\bfq_2$ be two partitions of $n$ such that $\bfq_1 \preceq \bfq_2$. Then for any $r \in (-\half, \half]$, we have
	\[ \lVert \xi_{r}(\bfq_1) \rVert \leq \lVert \xi_{r}(\bfq_2) \rVert. \]
	Moreover, when $r \notin \{0, \half\}$, the equality holds if and only if $\bfq_1 = \bfq_2$.
\end{lemma}

\begin{proof}
	Set $s(\bfq_i):=\lVert \xi_{r}(\bfq_i) \rVert^2$, $i=1,2$. Let $L$ be the maximum of numbers of columns of $\bfq_i$, $i=1,2$. Write the columns of $\bfq_i$ as $c_1(\bfq_i) \geq c_2(\bfq_i) \geq \cdots c_{L}(\bfq_i)$, $i=1,2$. Here we allow $c_j(\bfq_i) = 0$ for some $j$. Define the non-decreasing sequence $a_1, \ldots, a_L$ by 
		\[ a_j:= \left( r + (-1)^{j-1+\lfloor r \rfloor} \left\lfloor \frac{j}{2} \right\rfloor \right) ^2.\]
	Set $C_j(\bfq_i) := \sum_{k=1}^j c_k(\bfq_i)$ for $j\geq 0$ and $i=1,2$ (note that $C_0(\bfq_i)=0$).
	Then summation by parts gives
		\begin{equation} \label{eq:s(bfq)}
			s(\bfq_i) 
			= \sum_{j=1}^L c_j(\bfq_i) a_j
			= C_L(\bfq_i)a_L + \sum_{j=1}^{L-1} C_j(\bfq_i) (a_j - a_{j+1})
			= n a_L + \sum_{j=1}^{L-1} C_j(\bfq_i) (a_j - a_{j+1}).			 
		\end{equation}
	Therefore
	\[ s(\bfq_1)-s(\bfq_2) = \sum_{j=1}^{L-1} (C_j(\bfq_1)-C_j(\bfq_2)) (a_j-a_{j+1}).\]
	Since $\bfq_1 \preceq \bfq_2$, taking transpose gives $\bfq_1^t \succeq \bfq_2^t$ and hence $C_j(\bfq_1)-C_j(\bfq_2) \geq 0$ for all $j$. On the other hand, $a_j - a_{j+1} \leq 0$ since $a_j$ is non-decreasing. These two facts together with \eqref{eq:s(bfq)} gives $s(\bfq_1) \leq s(\bfq_2)$.

	When $r \notin \{0, \half\}$, the sequence $a_j$ is strictly increasing. Therefore the equality $s(\bfq_1) = s(\bfq_2)$ holds if and only if $C_j(\bfq_1) = C_j(\bfq_2)$ for all $j$, which is equivalent to $\bfq_1 = \bfq_2$.
\end{proof}

\begin{theorem}\label{thm:type_A}
	Let $\ckfg = \gl(n)$ and $\check{\Orb} = \check{\Orb}_\bfq$ be a nilpotent orbit in $\ckfg$ corresponding to a partition $\bfq$ of $n$. Given any $r \in (-\half, \half]$, let $\lambda_{\check{\Orb}, r} = \xi_{r}(\bfq) \in \ckfh/W$ be the infinitesimal character defined above. Then the maximal primitive ideal $J_{max}(\lambda_{\check{\Orb}, r})$ for $\fg = \gl(n)$ is mildly and hence weakly unipotent with respect to $\GL(n)$.
\end{theorem}

\begin{proof}
	This follows from \Cref{lem:gl_min}, \Cref{lem:gl_compare} and \Cref{cor:closure}.
\end{proof}

\begin{corollary}\label{cor:type_A_rtuple}
	Let $\ckfg = \gl(n)$ and $\check{\Orb} = \check{\Orb}_\bfq$ be a nilpotent orbit in $\ckfg$ corresponding to a partition $\bfq = [q_1, \ldots, q_l]$ of $n$. Given any $l$-tuple $\bfr = (r_1, \ldots, r_l)$ such that $r_i \in (-\half, \half]$ for all $i$, let $\lambda = \lambda_{\check{\Orb}, \bfr} = \xi_{\bfr}(\bfq) \in \ckfh/W$ be the infinitesimal character defined above. Then the maximal primitive ideal $J_{max}(\lambda_{\check{\Orb}, \bfr})$ for $\fg = \gl(n)$ is mildly unipotent with respect to $\GL(n)$.
\end{corollary}

\begin{proof}
	Let $s_1, \ldots, s_k$ be the distinct numbers in $\{r_1, \ldots, r_l\}$. Group $q_i$'s according to the values of $r_i$'s, i.e., write $\bfq = \bfq_1 \cup \bfq_2 \cup \cdots \cup \bfq_k$, where $\bfq_j$ is the sub-partition of $\bfq$ consisting of all $q_i$'s with $r_i = s_j$. Then $\xi_{\bfr}(\bfq) = \xi_{s_1}(\bfq_1) \cup \xi_{s_2}(\bfq_2) \cup \cdots \cup \xi_{s_k}(\bfq_k)$ up to permutation of coordinates. 
	
	Set $\Lambda := \lambda + Q'$. Then the associated pseudo-levi subalgebra $\fg_\Lambda$ is the standard Levi subalgebra $\gl(n_1) \times \gl(n_2) \times \cdots \times \gl(n_k)$, where $n_j$ is the size of the partition $\bfq_j$, $1 \leq j \leq k$. Let $\check{\Orb}_j$ be the nilpotent orbit in $\gl(n_j)$ corresponding to the partition $\bfq_j$. The rest follows from the same argument as in the proof of \Cref{thm:type_A} using \Cref{cor:closure}.
\end{proof}

\subsection{More general cases} \label{subsec:general_cases}

\begin{definition} \label{defn:rho}
	Suppose $\bfq = [q_1, q_2, \ldots, q_m]$ is a partition of $n$ and $\bfs = (s_1, s_2, \ldots, s_m)$ is an $m$-tuple of real numbers, satisfying $|s_i| < \half$ for all $i$. Define an $n$-tuple $\rho_{\bfr}(\bfq) \in \R^n$ by appending the sequences
		\[ \left(  \frac{q_i - 1}{2} + s_i, \frac{q_i - 3}{2} + s_i, \ldots, \frac{3 - q_i}{2} + s_i, \frac{1 - q_i}{2}+s_i  \right) \]
	for each $i \geqslant 1$. When $s_i = 0$ for all $i$, we simply write $\rho(\bfq) = \rho_{\bfzero}(\bfq)$. 

	We say that $(\bfq, \bfs)$ is \emph{antisymmetric} if the members of $\bfq$ and the coordinates of $\bfs$ can be permuted simultaneously so that $\bfq$ is of the form
	\begin{equation} \label{eq:antisymm_q}
		\bfq = [\bfd, \bfp_1, \bfp_1, \bfp_2, \bfp_2, \ldots, \bfp_l, \bfp_l],
	\end{equation} 
	where $\bfd$ and $\bfp_i$ are subpartitions of $\bfq$ with $\# \bfd = k $ and $\# \bfp_i = m_i$, whereas $\bfs$ is of the form
	\begin{equation} \label{eq:antisymm_s}
		\bfs = (\underbrace{0, \ldots, 0}_{k}, \underbrace{t_1, \ldots, t_1}_{m_1}, \underbrace{-t_1, \ldots, -t_1}_{m_1}, \underbrace{t_2, \ldots, t_2}_{m_2}, \underbrace{-t_2, \ldots, -t_2}_{m_2}, \ldots, \underbrace{t_l, \ldots, t_l}_{m_l}, \underbrace{-t_l, \ldots, -t_l}_{m_l})
	\end{equation}
	where $t_i$ are distinct positive real numbers. In particular, this implies that if $s_i \neq 0$ for some $i$, then $q_i$ has even multiplicity in $\bfq$.
\end{definition}

\begin{remark} \label{rem:rho_xi}
	Also note that, if we define a new $m$-tuple $\bfr = (r_1, \ldots, r_m)$ from $\bfs$ by setting, for each $1 \leq i \leq m$,
	\[ r_i = \begin{cases}
		s_i, & \text{ if } q_i \text{ is odd}, \\
		s_i - (-1)^{\lfloor s_i \rfloor}\half, & \text{ if } q_i \text{ is even},
	\end{cases} \]
	then $\rho_{\bfs}(\bfq)$ equals $\xi_{\bfr}(\bfq)$ in the sense of \Cref{defn:xi_r}.
\end{remark}

\begin{remark} \label{rem:antisymm}
	Observe that, if $(\bfq, \bfs)$ is antisymmetric, then $\rho_{\bfs}(\bfq) \in \R^n$ lies in the Cartan subalgebra of a classical Lie algebra. More precisely, $\rho_{\bfs}(\bfq)$ lies in the Cartan subalgebra of $\so(2n+1)$ (resp. $\sp(2n)$, $\so(2n)$) if $\bfq$ is a partition of $2n+1$ (resp. $2n$, $2n$) and the number of odd members of $\bfq$ is odd (resp. even, even).
\end{remark}

\begin{theorem} \label{thm:general_classical}
	Let $\fg$ be a classical Lie algebra and $\ckfg$ be the Langlands dual of $\fg$.
	With the notations in \Cref{defn:rho}, let $\check{\Orb} = \check{\Orb}_\bfq$ be a nilpotent orbit in $\ckfg$ corresponding to a partition $\bfq$ of $n$. Suppose $(\bfq, \bfs)$ is antisymmetric.	
	Let $\lambda = \lambda_{\check{\Orb}, \bfs} := \rho_{\bfs}(\bfq) \in \ckfh/W$ be the infinitesimal character defined above. Then the maximal primitive ideal $J_{max}(\lambda_{\check{\Orb}, \bfs})$ of $\Ug$ is mildly and hence weakly unipotent with respect to the root lattice of $\fg$.
\end{theorem}

\begin{proof}
	When $\fg$ is of type $A$, this is a special case of \Cref{cor:type_A_rtuple}. So we only need to consider the cases when $\fg$ is of type $B$, $C$ or $D$. We only illustrate the case when $\fg$ is of type $B$ or $D$. The case of type $C$ is along the same lines.

	Suppose $\ckfg = \mathfrak{so}(N)$, where $N \geqslant 7$. Let $n = \lfloor \frac{N}{2} \rfloor$. Write $\bfq$ as in \eqref{eq:antisymm_q} and $\bfs$ as in \eqref{eq:antisymm_s}. 
	As in the proof of \Cref{thm:q-unipotent_weaklyunip}, set $\bfd_i$ to be the subpartition of $\bfd$ consisting of all members of $\bfd$ that are congruent to $i$, for $i = 0, 1$. Set $N_i = |\bfd_i|$, $i = 0, 1$. 
	Then after conjugation, we can assume that the pseudo-levi subalgebra $\ckfg_\Lambda$ associated to the integral root datum equals the standard pseudo-levi subalgebra 
		\[ \ckfg_\Lambda = \so(N_0) \oplus \so(N_1) \oplus \bigoplus_{i=1}^l \gl(n_i),\]
	where $n_i = |\bfp_i|$, $1 \leq i \leq l$. Then $\lambda = \rho_{\bfs}(\bfq)$ can be regarded as the concatenation of $\lambda_0 = \rho_{\bfzero}(\bfd_0)$, $\lambda_1 = \rho_{\bfzero}(\bfd_1)$ and $\rho_{t_i}(\bfp_i) = \xi_{\bfr_i}(\bfp_i)$, $1 \leq i \leq l$ (up to conjugation by the Weyl group $S_{N-1}$ of $\sl(N)$), where each $\bfr_i$ is an $m_i$-tuple obtained from $(t_1, \cdots, t_1)$ (repeated $m_i$ times) as in Remark \ref{rem:rho_xi}. 

	Now the Richardson orbit $\Ind_{\ckfg_\lambda}^{\ckfg_\Lambda} \bfzero$ is the product of corresponding Richardson orbits in each simple factor of $\ckfg_\Lambda$. More precisely, the Richardson orbit in $\so(N_0)$ (resp. $\so(N_1)$) is $f_{CD}(\check{\Orb}_{\bfd_0})$ (resp. $\check{\Orb}_{\bfd_1}$) as in the proof of \Cref{thm:q-unipotent_weaklyunip}, whereas the Richardson orbit in each $\gl(n_i)$-factor is the regular nilpotent orbit. 
	We can then apply the same argument as in the proof of \Cref{thm:q-unipotent_weaklyunip} to the two $\so$-factors, and apply \Cref{thm:type_A} to each $\gl(n_i)$-factor as in the proof of \Cref{cor:type_A_rtuple}.
\end{proof}

\begin{corollary} \label{cor:cover_weaklyunip_classical}
	Let $\g$ be a classical Lie algebra. Then the unipotent ideal $I_0(\widetilde{\Orb})$ attached to a connected cover $\widetilde{\Orb}$ of any nilpotent orbit in $\fg^*$ is mildly and hence weakly unipotent with respect to the root lattice of $\g$.
\end{corollary}

\begin{proof}
	When $\g$ is of type $B$, $C$ or $D$ and the cover $\widetilde{\Orb}$ is equivariant with respect to the linear classical group $G$ corresponding to $\g$, it follows from \cite[Proposition 8.2.8, (ii), (iii)]{LMBM} that the infinitesimal character $\gamma_0(\widetilde{\Orb})$ is $q$-unipotent.

	When $\g = \sl(N)$, or $\g=\so(N)$ and $\widetilde{\Orb}$ is equivariant with respect to $\Spin(N)$ but not $\SO(N)$, we only need to check that the infinitesimal character $\gamma_0(\widetilde{\Orb})$ is of the form $\rho_{\bfs}(\bfq)$ for some antisymmetric pair $(\bfq, \bfs)$ in the sense of \Cref{defn:rho}, then apply \Cref{thm:general_classical}.
	When $\g = \sl(N)$, this follows from \cite[Proposition 7.6.4]{LMBM}. When $\g = \so(N)$, this follows from \cite[Proposition 4.2.6]{MBM} and its proof.
\end{proof}

\section{The case of exceptional groups} \label{sec:exceptional}

Instead of checking that all unipotent ideals for simple exceptional Lie algebra are weakly unipotent, we describe a very simple algorithm to check the mild unipotence of the unipotent ideals $J_{max}(\lambda)$ attached to birationally rigid covers for exceptional simple Lie algebras, as studied in \cite[\S\, 4.3.1 \& 4.3.2]{MBM}. This is already enough for the applications in \cite{Davis-Mason-Brown:Hodge}, see Corollary 5.23 and Theorem 5.25 there.
We have implemented this algorithm in the \texttt{atlas} software. 

Let $\g$ be an exceptional simple Lie algebra. Let $\lambda \in \fh^*_\R$ be the infinitesimal character of the unipotent ideal $I(\widetilde{\Orb}) = J_{max}(\lambda)$ attached to a birationally rigid cover $\widetilde{\Orb}$ of a nilpotent orbit $\Orb$ in $\fg^*$. For any $\nu \in \Lambda$, let $n_\nu$ be the number of roots of the reductive subalgebra $\ckfg_\nu = \mathfrak{z}_{\ckfg}(\nu)$. 
We run through all $\nu \in \Lambda$ such that $\lVert \nu \rVert \leq \lVert \lambda \rVert$, and check whether $n_\nu > n_\lambda$. If this is the case, then 
\[
	\dim \left( \Ind_{\ckfg_\lambda}^{\ckfg_\Lambda} \bfzero \right)  > \dim \left( \Ind_{\ckfg_\nu}^{\ckfg_\Lambda} \bfzero \right)
\] 
and hence
\[
	\Ind_{\ckfg_\lambda}^{\ckfg_\Lambda} \bfzero  \not\preceq \Ind_{\ckfg_\nu}^{\ckfg_\Lambda} \bfzero.  
\]
By \Cref{cor:closure}, this implies that the unipotent ideal $I(\widetilde{\Orb})$ is mildly unipotent, with respect to the choice of $X^*$. We have checked this for both root lattice and weight lattice for all simple exceptional Lie algebras and all unipotent ideals attached to birationally rigid covers.

\bibliographystyle{alphaurl}
\bibliography{duality}

\end{document}